\documentclass[a4paper]{article}

\usepackage{amsfonts, latexsym, amssymb, amsthm, amsmath, mathrsfs, esint}

\newcommand{\R}{\mathbb{R}}
\newcommand{\N}{\mathbb{N}}

\newcommand{\M}{\mathbf{M}}

\newcommand{\vol}{\mathrm{vol}}

\newcommand{\ndim}{{\bar{n}}}
\newcommand{\Lin}{\mathcal{L}}
\newcommand{\transpose}{\mathsf{T}}
\newcommand{\sff}{\mathrm{I{\mkern -1mu}I}}

\let\div\relax\DeclareMathOperator{\div}{div}
\DeclareMathOperator*{\esssup}{ess \, sup}
\DeclareMathOperator{\tr}{tr}
\DeclareMathOperator{\dist}{dist}
\DeclareMathOperator{\supp}{supp}
\DeclareMathOperator{\im}{im}
\DeclareMathOperator{\grad}{grad}
\DeclareMathOperator{\diam}{diam}
\newcommand{\blank}{{\mkern 2mu\cdot\mkern 2mu}}
\newcommand{\so}{\mathfrak{so}}

\newcommand{\scp}[2]{\left\langle #1, #2 \right\rangle}
\newcommand{\pairing}[2]{\langle #1; #2 \rangle}
\newcommand{\smallscp}[2]{\langle #1, #2 \rangle}
\newcommand{\dd}[2]{\frac{\partial #1}{\partial #2}}
\newcommand{\set}[2]{\left\{ #1 \colon #2 \right\}}

\newcommand{\restr}{\mathchoice
{\kern2pt\mbox{\vrule width 0.08ex height1.5ex depth0ex\kern-0.08ex\vrule width 1.5ex height.08ex depth0ex}\kern2pt}
{\kern2pt\mbox{\vrule width 0.08ex height1.5ex depth0ex\kern-0.08ex\vrule width 1.5ex height.08ex depth0ex}\kern2pt}
{\kern1.5pt\mbox{\vrule width 0.06ex height1.1ex depth0ex\kern-0.06ex\vrule width 1.1ex height.06ex depth0ex}\kern1.5pt}
{\kern1pt\mbox{\vrule width 0.04ex height0.75ex depth0ex\kern-0.04ex\vrule width 0.75ex height.04ex depth0ex}\kern1pt}
}

\newtheorem{theorem}{Theorem}
\newtheorem{lemma}[theorem]{Lemma}
\newtheorem{proposition}[theorem]{Proposition}
\newtheorem{corollary}[theorem]{Corollary}

\theoremstyle{definition}
\newtheorem{definition}[theorem]{Definition}

\theoremstyle{remark}
{}

\newtheorem*{acknowledgement}{Acknowledgement}

\begin{document}

\title{A characterisation of $\infty$-harmonic maps in terms of $1$-currents}

\author{Roger Moser\footnote{Department of Mathematical Sciences,
University of Bath, Bath BA2 7AY, UK.
E-mail: r.moser@bath.ac.uk}}

\maketitle

\begin{abstract}
We consider maps between two Riemannian manifolds and study a functional
given in terms of the $L^\infty$-norm of the derivative.
This functional is not differentiable, but we can define critical points with
the help of a subdifferential. The resulting notion includes, for example, minimisers in a given
homotopy class.

We derive a geometric condition equivalent to criticality in this sense.
The condition is formulated in terms of a vector-valued $1$-current on
the domain manifold, which encapsulates some of the key properties of
the critical point. Moreover, this $1$-current is itself a critical point of a
generalised mass functional.
\end{abstract}

\section{Introduction}

Let $M$ and $N$ be two compact Riemannian manifolds without boundary, of dimension $m$ and $n$, respectively.
We consider maps $u \colon M \to N$ and study a variational problem involving a functional of the form
\[
E_\infty(u) = \esssup_{x \in M} \sqrt{F(du(x))}
\]
for a suitable convex, quadratically homogeneous function $F$. We will define a notion of critical points for
$E_\infty$ below, which may, however, look unfamiliar because $E_\infty$ is not differentiable,
not even in the Gateaux sense, no matter how smooth $F$ is chosen.

The reader may think of the two typical cases $F(du) = \frac{1}{2} |du|_2^2$ and
$F(du) = \frac{1}{2} |du|_\infty^2$, where $|\blank|_2$ denotes the Hilbert-Schmidt norm and $|\blank|_\infty$ denotes the operator norm induced by the Riemannian metrics on $M$ and $N$. Indeed, in the introduction, we formulate our results for these two
examples, but our analysis is more general, and we also allow $F(du) = \frac{1}{2} \|du\|^2$ for any orthogonally
invariant norm $\|\blank\|$ later on. For example, this includes the $p$-Schatten norms used by
Daskalopoulos and Uhlenbeck \cite{Daskalopoulos-Uhlenbeck:24} in a similar context.

When $F(du) = \frac{1}{2} |du|_2^2$, we have the $L^\infty$-counterpart of
the problem of harmonic maps \cite{Eells-Lemaire:78, Eells-Lemaire:88}, or more generally, $p$-harmonic maps \cite{Wei:08}.
These are the critical points of the functionals
\[
E_2(u) = \left(\frac{1}{2} \int_M |du|_2^2 \, d\vol\right)^{\frac{1}{2}} \quad \text{and} \quad E_p(u) = \left(\frac{1}{p} \int_M |du|_2^p \, d\vol\right)^{\frac{1}{p}},
\]
respectively, where $\vol$ denotes the volume measure on $M$ induced by the Riemannian metric. The choice $F(du) = \frac{1}{2} |du|_\infty^2$, on the other
hand, gives rise to the local Lipschitz constant for $u$ (up to a constant),
which is an important quantity in the theory of optimal Lipschitz extensions \cite{Aronsson:67, Jensen:93, Crandall-Evans-Gariepy:01, Sheffield-Smart:12}.

For variational problems involving the functional $E_\infty$, the natural space to work in is $C^{0, 1}(M; N)$. (It can alternatively be identified
with a Sobolev space $W^{1, \infty}(M; N)$, which is defined later.)
The global minimisers are constant maps and are thus not particularly
interesting. We therefore want to study minimisers under additional conditions,
such as in a prescribed homotopy class.
(Minimisers in a given homotopy class always exist and can be constructed
with the direct method, in contrast to harmonic or $p$-harmonic maps.)
More generally, we are interested in critical points.
But because of the lack of differentiability, it is not immediately clear how to understand that term.
We use the notation $\dist_N$ for the geodesic distance function on $N$, and
then we equip $C^{0, 1}(M; N)$ with the uniform distance metric (\emph{not}
the $C^{0, 1}$-metric), defined by
\[
\dist_{C^0(M; N)}(u, v) = \sup_{x \in M} \dist_N(u(x), v(x)).
\]
Then we have the following notion.

\begin{definition} \label{def:infinity-harmonic}
A map $u \in C^{0, 1}(M; N)$ is called \emph{$\infty$-harmonic} if
\[
\liminf_{v \to u} \frac{E_\infty(v) - E_\infty(u)}{\dist_{C^0(M; N)}(u, v)} \ge 0.
\]
\end{definition}

Thus we use a variant of the Fr\'echet subdifferential \cite{Kruger:03},
even though the underlying space is neither a vector space nor complete with this metric.
(For any $C > 0$, however, the set of all $u \in C^{0, 1}(M; N)$ with $E_\infty(u) \le C$, is
compact by the theorem of Arzel\`a-Ascoli.)
Any local minimiser of $E_\infty$ with respect to the uniform
distance is clearly an example of an $\infty$-harmonic map in this sense.
This applies in particular to a minimiser in a fixed homotopy class.
The use of the metric $\dist_{C^0(M; N)}$ may seem somewhat
arbitrary at first, but arises quite naturally from our subsequent analysis.

Remarkably, the definition automatically gives rise to a stronger condition.

\begin{proposition} \label{prp:energy-discrepancy}
There exists a constant $C \ge 0$ such that for any $\infty$-harmonic map $u \in C^{0, 1}(M; N)$,
the inequality
\[
E_\infty(v) \ge E_\infty(u) \bigl(1 - C \dist_{C^0(M; N)}^2(u, v)\bigr)
\]
is satisfied for all $v \in C^{0, 1}(M; N)$.
\end{proposition}

In other words, not only is the energy discrepancy $E_\infty(v) - E_\infty(u)$ bounded quadratically
with the distance from $u$, but the constant in the corresponding inequality is independent of $u$.
This is clearly useful information when we study the set of all $\infty$-harmonic maps.
Indeed, the following is an immediate consequence.

\begin{corollary} \label{cor:compact}
Let $C \ge 0$. Then the set of all $\infty$-harmonic maps $u \in C^{0, 1}(M; N)$ with
$E_\infty(u) \le C$, is compact.
\end{corollary}

We think of Definition \ref{def:infinity-harmonic} as a way to extend the idea of critical points to a functional
involving the essential supremum instead of an integral. Because of the lack of differentiability,
it may seem surprising that the concept gives rise to differential
equations, but it has been shown that this is the case for similar definitions in the context of
other problems \cite{Moser:22, Katzourakis-Moser:23, Gallagher-Moser:23, Gallagher-Moser:24}.
In a similar vein, we will derive some conditions equivalent to the criterion
from Definition \ref{def:infinity-harmonic} (and will prove
some additional properties), but they are somewhat different in structure
from the aforementioned papers and will require
different arguments. The methods in this paper are more similar to those
from  a paper of Katzourakis and Moser \cite{Katzourakis-Moser:25},
which analyses maps from a domain in $\R^m$ to $\R^n$ that minimise
$E_\infty$ globally (for $F(du) = \frac{1}{2}|du|_2^2$ only) for prescribed boundary data.
As explained previously, global minimisers are not interesting in the
context studied here, and this is one reason why we need to extend and refine
those arguments. We also consider more general functions $F$, and of course,
the geometry of the manifolds $M$ and $N$ (especially the latter)
will give rise to other additional phenomena and difficulties.

One key property of a functional such as $E_\infty$, and one of the difficulties
when we want to develop a variational theory for it, is the fact that
many variations of $u$ leave it invariant. For example, if $u \in C^1(M; N)$,
then we may consider the open set $\Omega = \set{x \in M}{F(du(x)) < (E_\infty(u))^2}$.
Any sufficiently small variation supported in $\Omega$ will then not affect
the value of $E_\infty$. Therefore, a suitable theory will have to
identify a part of $M$ where the behaviour of $u$ matters. In our results, we will be able to say something
about the geometry of an object that may be regarded as a collection of
generalised curves in $M$, and about how it interacts with $u$.

Before we can formulate our results, however, we need some technical tools.
First, we will use an extension of the standard concept of a $1$-current from geometric measure theory.
(A discussion of currents can be found, e.g., in a book by Simon \cite{Simon:83}.)
An \emph{$\R^n$-valued $1$-current} on $M$ is simply an $n$-tuple $T = (T_1, \dotsc, T_n)$
of $1$-currents on $M$. For an $n$-tuple of $1$-forms $\sigma = (\sigma_1, \dotsc, \sigma_n)$ (also called an $\R^n$-valued $1$-form), we then write
$T(\sigma) = \sum_{i = 1}^n T_i(\sigma_i)$. The \emph{mass} of $T$ is
\[
\M(T) = \sup_{\|\sigma\|_{C^0(M)} \le 1} T(\sigma).
\]
If $\M(T) < \infty$, then there exist a Radon measure $\|T\|$ on $M$ and a $\|T\|$-measurable section $\vec{T} = (\vec{T}_1, \dotsc, \vec{T}_n)$ of $TM \otimes \R^n$ with $|\vec{T}|_2 = 1$ almost everywhere, such that
\[
T(\sigma) = \int_M \pairing{\sigma}{\vec{T}} \, d\|T\|
\]
for all $\R^n$-valued $1$-forms $\sigma$ on $M$, where
$\pairing{\sigma}{\vec{T}} = \sum_{i = 1}^n \pairing{\sigma_i}{\vec{T}_i}$
denotes the duality pairing between $n$-tuples of $1$-forms and $n$-tuples of
tangent vectors.
Both $\|T\|$ and $\vec{T}$ are uniquely determined by $T$ subject to the
usual identification of functions that coincide almost everywhere.

Rather than the usual boundary operator for $1$-currents, we use a variant adapted to a given map $u \colon M \to N$. To this end, we make two assumptions on $N$: first, that it is isometrically embedded in a Euclidean space $\R^\ndim$; and second, that there exist smooth vector fields $\bar{\eta}_1, \dotsc, \bar{\eta}_n$ on $N$ that form an orthonormal basis of $T_y N$ at every point $y \in N$ (i.e., that $N$ is parallelisable).
Neither assumption restricts the
generality of our analysis by the Nash embedding theorem \cite{Nash:56}
and by the arguments of H\'elein \cite{Helein:91.1}.
For the time being, we also assume that $u \in C^1(M; N)$, although for a reasonably rich theory of the variational problem,
this is too strong an assumption. (Eventually we will formulate a variant of
the idea that does not require this assumption.)
We set $\eta_i = \bar{\eta}_i \circ u \colon M \to \R^\ndim$ and consider the $\so(n)$-valued $1$-form $\omega$ with components
$\omega_{ij} = d\eta_i \cdot \eta_j$ (where $\cdot$ denotes the inner
product in $\R^\ndim$). Then we write
\[
\partial^u T(\theta) = T(d\theta - \omega \theta)
\]
for any $\theta \in C^\infty(M; \R^n)$.
We call this the \emph{$u$-boundary} of $T$. (The operator $\partial^u$ is
thus the adjoint of the covariant derivative $d - \omega$ on the trivial
vector bundle $M \times \R^n$.
To understand our results, however, it is not necessary to consider that connection.)

In this introduction, we formulate simplified versions of our main results.
The full versions, which are given in Section \ref{sct:main-results} below, require a significant
number of other technical tools before we can write them down, but the following can
still give a good idea of the content. We therefore restrict our attention to the two examples $F(du) = \frac{1}{2} |du|_2^2$ and $F(du) = \frac{1}{2} |du|_\infty^2$. Thus
\begin{equation} \label{eq:functional-simplified}
E_\infty(u) = \frac{1}{\sqrt{2}} \esssup_{x \in M} |du|_p,
\end{equation}
where $p = 2$ or $p = \infty$. As mentioned before, we also assume that
we have a map $u \in C^1(M; N)$ at first, but will still
allow comparison maps in $C^{0, 1}(M; N)$ for the purpose of
Definition \ref{def:infinity-harmonic}.

We write $\scp{\blank}{\blank}_M$ for the Riemannian metric on $M$.
We further use the notation $\grad u$ for the component-wise gradient
of the map $u \colon M \to N \subseteq \R^\ndim$.

\begin{theorem} \label{thm:geodesic-simplified}
Let $p = 2$, and let $E_\infty$ be the corresponding functional given by \eqref{eq:functional-simplified}. Suppose that $u \in C^1(M; N)$ is $\infty$-harmonic and that $e_\infty = E_\infty(u) > 0$.
Then there exists an $\R^n$-valued $1$-current $T \neq 0$ on $M$ of finite mass such that $\partial^u T = 0$ and such that
\begin{equation} \label{eq:T-represents-du-simplified}
\grad u = \sqrt{2} e_\infty \sum_{i = 1}^n \vec{T}_i \otimes \eta_i
\end{equation}
at $\|T\|$-almost every point of $M$ and
\begin{equation} \label{eq:stationary-simplified}
\sum_{i = 1}^n \int_M \smallscp{\vec{T}_i}{\nabla_{\vec{T}_i} \xi}_M \, d\|T\| = 0
\end{equation}
for any smooth vector field $\xi$ on $M$.
\end{theorem}

We should think of $T$ as a geometric object that carries essential information
about $u$. In particular, equation \eqref{eq:T-represents-du-simplified} determines the
derivative $du$ on $\supp T$ in terms of $\vec{T}$ and $e_\infty$. Moreover,
it follows that $\supp T$ is contained in the set where $|du|_2$ attains its
maximum. The condition $\partial^u T = 0$ is geometric in nature, but depends
on $u$. Equation \eqref{eq:stationary-simplified}, on the other hand, does
not depend on $u$. This is also a geometric condition; indeed, the left-hand
side of the equation is the first variation of the functional $\M$ under variations induced by the vector field $\xi$. Thus $T$ is a critical point of $\M$ in this sense.

Remarkably, our next theorem implies that the conjunction of $\partial^u T = 0$
with equation \eqref{eq:T-represents-du-simplified}, for $e_\infty = E_\infty(u)$,
is in fact \emph{equivalent} to $u$ being
$\infty$-harmonic. Thus this pair of conditions may be thought of as
tantamount to an Euler-Lagrange equation, and equation \eqref{eq:stationary-simplified} as additional information.

\begin{theorem} \label{thm:equivalence-simplified}
Let $p = 2$, and let $E_\infty$ be the functional given by \eqref{eq:functional-simplified}. Suppose that
$u \in C^1(M; N)$ and $e_\infty \ge E_\infty(u) > 0$, and that
there exists an $\R^n$-valued $1$-current $T \neq 0$ on $M$ of finite mass such that
$\partial^u T = 0$ and such that \eqref{eq:T-represents-du-simplified} holds almost everywhere with respect to $\|T\|$.
Then $u$ is $\infty$-harmonic and $E_\infty(u) = e_\infty$.
\end{theorem}

For $p = \infty$, we have weaker results, mostly due to the lack of strong convexity of the function $|\blank|_\infty^2$. We can, however, still prove the following. Here we write $|A|_1$ for the trace norm (the sum of the singular values) of an element of $T_x M \otimes \R^n$ for any $x \in M$.

\begin{theorem} \label{thm:weak-geodesic-simplified}
Let $p = \infty$, and let $E_\infty$ be the corresponding functional given by \eqref{eq:functional-simplified}. Suppose that $u \in C^1(M; N)$ is $\infty$-harmonic, and that $e_\infty = E_\infty(u) > 0$.
Then there exists an $\R^n$-valued $1$-current $T \neq 0$ on $M$ of finite mass such that
$\partial^u T = 0$ and such that
\begin{equation} \label{eq:T-almost-determines-du-simplified}
\sum_{i = 1}^n du(\vec{T}_i) \cdot \eta_i = \sqrt{2}e_\infty |\vec{T}|_1
\end{equation}
almost everywhere with respect to $\|T\|$.
\end{theorem}

Equation \eqref{eq:T-almost-determines-du-simplified} does not quite determine
$du$ on $\supp T$ (and the reason is once more the lack of strong convexity
of $|\blank|_\infty^2$), but still gives a strong restriction. Our methods
do not give any information of the form \eqref{eq:stationary-simplified} in this situation, nor do we have a counterpart of Theorem \ref{thm:equivalence-simplified}.
It is an open question whether this can be expected.

As previously mentioned, we will state and prove versions of all of these
results without the assumption that $u \in C^1(M; N)$ and for more general choices of $F$. Indeed, Theorems \ref{thm:geodesic-simplified}--\ref{thm:weak-geodesic-simplified} are then just obvious corollaries of the more
general results, which can be found in Section \ref{sct:main-results}.

All of the results in this paper have an obvious extension to the case
where $N$ is complete, but not necessarily compact. After all, once we
fix $u \in C^{0, 1}(M; N)$, the image $u(M)$ is compact, and $N$ may be
modified elsewhere to make it compact. We keep the compactness assumption
for the rest of the paper, however, because it is occasionally convenient.

We conclude the introduction by explaining the organisation of the paper.
In the next section, we provide a more extensive discussion of previous
work and known results related to $\infty$-harmonic maps. Then, in Section
\ref{sct:tools}, we give the definitions and explain the notation that
we require for the full versions of our main results.
These results are formulated in Section \ref{sct:main-results}.
In Sections \ref{sct:properties-of-F} and \ref{sct:vector-fields}, we prepare
the ground for the proofs by deriving some properties of the function $F$
in the definition of $E_\infty$, and by discussing further useful tools.

The proofs of Theorem \ref{thm:geodesic-simplified} and \ref{thm:weak-geodesic-simplified},
and of the corresponding results in Section \ref{sct:main-results},
require the construction on an $\R^n$-valued $1$-current $T$ with certain
properties. The construction relies on the approximation of the
$\infty$-harmonic map $u$ by $p$-harmonic maps (up to a penalisation term),
so that we can make use of an Euler-Lagrange equation. These arguments are
carried out in Section \ref{sct:construction}. The proof of Proposition \ref{prp:energy-discrepancy}
follows from some of the same arguments, practically as a by-product. For some of our results, we
then need to examine the convergence as $p \to \infty$ in more detail,
which we do in Section \ref{sct:strong-convergence}.

The proof of Theorem \ref{thm:equivalence-simplified} and its generalisation
is given in Section \ref{sct:equivalence}.

Since Theorem \ref{thm:geodesic-simplified} involves the first variation
of the mass functional $\M$, which is replaced by something else for other
choices of $F$, we also make some statements about variations of
$\R^n$-valued currents in the course of this paper. We prove such statements
in Section \ref{sct:deformations}. Finally, in Section \ref{sct:consistency},
we prove a technical result that is required to establish the connection
between Theorem \ref{thm:weak-geodesic-simplified} and its more general
version.

\section{The literature on $\infty$-harmonic functions and maps}

The problem studied in this paper is an example of a variational problem in
$L^\infty$. This theory was pioneered by Aronsson \cite{Aronsson:65, Aronsson:66, Aronsson:67, Aronsson:68, Aronsson:84}. There is now an extensive body of
literature on problems involving functions $u \colon \Omega \to \R$
for a domain $\Omega \subseteq \R^m$, of which we mention only a few additional
contributions \cite{Bhattacharya-DiBenedetto-Manfredi:89, Jensen:93, Savin:05, Evans-Savin:08, Evans-Smart:11.2}.

One of the key ideas in this theory is that of an absolute minimiser, which
is a function that not only minimises $\esssup_{\Omega} |du|$, but such that
$\esssup_{\Omega'} |du| \le \esssup_{\Omega'} |dv|$ for any suitably regular comparison
function $v \colon \Omega \to \R$ with $\supp(v - u) \subseteq \overline{\Omega'}$ and
for any suitably regular subset $\Omega' \subseteq \Omega$.
Absolute minimisers exist under reasonable assumptions, and they are the unique
viscosity solutions of the Aronsson equation
\[
\sum_{\alpha, \beta = 1}^m \dd{u}{x^\alpha} \dd{u}{x^\beta} \frac{\partial^2 u}{\partial x^\alpha \partial x^\beta} = 0.
\]
It is therefore sensible to consider this the Euler-Lagrange equation for the
variational problem.

The theory breaks down completely, however, if the codomain $\R$ is replaced
by the Euclidean space $\R^n$ (let alone a manifold). But the problem then also
has several interesting variants. Now $du(x)$ is a linear map between two
vector spaces for every point $x$ in the domain, and there is more than one reasonable way to
equip the space of such maps with a norm. There is some work by
Sheffield and Smart \cite{Sheffield-Smart:12} for the norm $|\blank|_\infty$
and by Katzourakis \cite{Katzourakis:12, Katzourakis:13, Katzourakis:14.1, Katzourakis:14.2, Katzourakis:15.2, Katzourakis:17.1} for $|\blank|_2$.
All of these papers extend some aspects of the
scalar-valued problem to the vector-valued one, but even the collection of
all of them does not amount to a complete theory. It is not known if
absolute minimisers exist, except in some special cases. A generalisation of the Aronsson equation
can be derived formally, and some statements can be made for the solutions,
especially under the assumption of sufficient regularity, but the connection
to the variational problem is tenuous.

There is another paper, by Katzourakis and Moser \cite{Katzourakis-Moser:25},
that takes a different approach. Rather than attempting to salvage pieces
of the scalar-valued theory, it explores some measure-theoretic ideas, which in a
more primitive form have their origin in work by Evans \cite{Evans:03.2} and by
Evans and Yu \cite{Evans-Yu:05}. Reinterpreted and refined, they give rise
to the vector-valued $1$-currents that are discussed in the introduction.
The present paper extends that theory and explores the question how the results and arguments
are affected when we change the geometry. But the different context necessitates (or at least
suggests) some additional changes as well.
\begin{enumerate}
\item \label{itm:critical-points}
Whereas the paper of Katzourakis and Moser focuses on minimisers, we should now consider critical points more generally,
as explained in the introduction.
\item \label{itm:general-functions}
Moreover, the previous paper considers only the
function $F(du) = \frac{1}{2} |du|_2^2$, but in various contexts, other choices are
just as interesting. (This applies to $F(du) = \frac{1}{2} |du|_\infty^2$
in particular
\cite{Sheffield-Smart:12, Daskalopoulos-Uhlenbeck:24, Daskalopoulos-Uhlenbeck:22, Daskalopoulos-Uhlenbeck:24.2}.)
\end{enumerate}

Our definition of critical points, and the strategy for addressing \ref{itm:critical-points},
is similar to some ideas introduced in another paper by the author \cite{Moser:22}, and also
explored in subsequent works \cite{Katzourakis-Moser:23, Gallagher-Moser:23, Gallagher-Moser:24}.
The main difference to previous results is best seen by comparing with
another paper of Katzourakis and Moser
\cite{Katzourakis-Moser:23}: when studying a variational problem for a quantity such as $\esssup |\mathcal{D} u|$
for some differential operator $\mathcal{D}$ (not necessarily linear), we are led to equations involving the $L^2$-adjoint
$\mathcal{D}^*$ (of the linearisation at the relevant points). If $\mathcal{D}^*$ provides good control of
the solutions, then this is extremely helpful. But this is not the case for the problem studied in this paper.

The question implied by \ref{itm:general-functions}, has also been studied in a recent paper by Ignat and Moser
\cite{Ignat-Moser:25}, albeit in a different setting and under different assumptions.
This is a paper about a problem from materials science, but at one
point a variational problem in $L^\infty$, depending on some convex function $F$, is important.
Using the vector-valued $1$-currents arising from it, the paper manages to estimate the
infimum of the functional in the appropriate space. A similar application is also conceivable
when we work with maps between manifolds. Indeed, there are variants of the application
studied by Ignat and Moser that involve maps into the unit sphere $S^2 \subseteq \R^3$ \cite{Ignat-Merlet:11}.

Variational problems in $L^\infty$ for maps between manifolds have also recently been studied by
Daskalopoulos and Uhlenbeck \cite{Daskalopoulos-Uhlenbeck:24, Daskalopoulos-Uhlenbeck:22, Daskalopoulos-Uhlenbeck:24.2}.
Their work is motivated by some ideas of Thurston \cite{Thurston:98} on best Lipschitz maps
between hyperbolic surfaces in the context of Teichm\"uller theory.
For this reason, much of the work of Daskalopoulos and Uhlenbeck focuses on hyperbolic
manifolds (but some of their results apply in greater generality, especially for the
approximations by $p$-harmonic maps). These papers also use ideas from
geometric measure theory not dissimilar to what is described in the introduction, and indeed
some of their results highlight different aspects of the same objects. But due to a
different formalism and different assumptions, the relationship between the
two theories is sometimes complicated and sometimes unclear.

There is also some less recent work on the generalisation of the Aronsson equation for
maps between Riemannian manifolds \cite{Troutman:08, Wang-Ou:09, Ou-Troutman-Wilhelm:12}, but no
connection to the variational problem is established.

\section{Tools and notation} \label{sct:tools}

We require several tools not just for our analysis, but even to state the
main results in full generality. Some of them are well-known, but perhaps not
at the forefront of the reader's mind. Others are constructed specifically
for our purpose.
In this section, we introduce these tools and the corresponding notation.

We begin with some notation for concepts from linear algebra and matrix analysis.
Given two finite-dimensional Hilbert spaces
$V$ and $W$, we write $\Lin(V; W)$ for the space of all linear maps $V \to W$.
Given $A, B \in \Lin(V; W)$, we use the notation $A : B = \tr (A B^*)$ for the Hilbert-Schmidt inner product, where $B^* \in \Lin(W; V)$ is the adjoint of $B$.
(It will be convenient to consider $AB^*$ rather than $A^*B$ in this
situation, because for most of the paper,
the space $V$ will vary with $x \in M$, whereas $W$ will remain fixed.)

The dual of $\Lin(V; W)$ can be identified with $\Lin(W; V)$ via the
duality pairing
$\pairing{\blank}{\blank} \colon \Lin(W; V) \times \Lin(V; W) \to \R$ given by 
$\pairing{B}{A} = \tr(AB)$ for $A \in \Lin(V; W)$ and $B \in \Lin(W; V)$.
We use the notation $A : B$ for the Hilbert-Schmidt inner product in
$\Lin(W; V)$ as well, i.e., we write $A : B = \tr(A^* B)$ when $A, B \in \Lin(W; V)$.

Suppose that $\ndim$ is the dimension of $W$. Then given $A \in \Lin(V; W)$,
we write $s(A) \in \R^\ndim$ for the $\ndim$-tuple
consisting of the square roots of the eigenvalues of $A A^*$ (i.e., the singular values of $A^*$), repeated according to their multiplicities
and in descending order (to make it well-defined, although the order is not important).
A \emph{symmetric gauge function} on $\R^\ndim$ is a norm
$f \colon \R^\ndim \to [0, \infty)$ such that $f(y_1, \dotsc, y_\ndim) = f(|y_1|, \dotsc, |y_\ndim|)$
for any $y \in \R^\ndim$ and $f(y_{\alpha(1)}, \dotsc, y_{\alpha(\ndim)}) = f(y_1, \dotsc, y_\ndim)$
for any permutation $\alpha \colon \{1, \dotsc \ndim\} \to \{1, \dotsc, \ndim\}$. Given a symmetric gauge function $f$,
We can define the orthogonally invariant norm $\|\blank\|_f$ on $\Lin(V; W)$ by setting $\|A\|_f = f(s(A))$.
This is discussed in more detail, e.g., in a book by Horn and Johnson \cite{Horn-Johnson:85}.

The dual of $f$ is another symmetric gauge function $f^* \colon \R^\ndim \to \R$,
given by the formula
\[
f^*(z) = \sup_{f(y) \le 1} y \cdot z, \quad z \in \R^\ndim.
\]
It is convenient to represent the dual of $\|\blank\|_f$ as a function on
$\Lin(W; V)$:
\[
\|B\|_f^* = \sup_{\|A\|_f \le 1} \pairing{B}{A}, \quad B \in \Lin(W; V).
\]
The identity $\|B\|_f^* = \|B^*\|_{f^*}$ is well-known
\cite[pp.\ 439--440]{Horn-Johnson:85}.

We are mostly interested in the squares of these norms. Suppose that $g \colon \R^\ndim \to [0, \infty)$ and
$F \colon \Lin(V, W) \to [0, \infty)$ are defined by $g(y) = \frac{1}{2} (f(y))^2$ and $F(A) = \frac{1}{2} \|A\|_f^2$.
Then we have the Legendre transforms
\[
g^*(z) = \sup_{y \in \R^\ndim} (y \cdot z - g(y)) = \frac{1}{2} (f^*(z))^2, \quad z \in \R^\ndim,
\]
and
\[
F^*(B) = \sup_{A \in \Lin(V; W)} (\pairing{B}{A} - F(A)) = \frac{1}{2} (\|B\|_f^*)^2 = \frac{1}{2}\|B^*\|_{f^*}^2, \quad B \in \Lin(W; V).
\]

If the function $g \colon \R^\ndim \to \R$ is differentiable at $y \in \R^\ndim$, we write $dg(y)$ for its Fr\'echet derivative.
For a function $F \colon \Lin(V; W) \to \R$ or $\Lin(W; V) \to \R$, we work predominantly with the gradient (if it exists),
denoted by $\nabla F$. Thus $\nabla F(A) : B = \frac{d}{dt}|_{t = 0} F(A + tB)$.

We will sometimes impose the following condition on $f$, which amounts to
$C^{1, 1}$-regularity and strong convexity of $g$ with uniform bounds.

\begin{definition} \label{def:regular}
Let $f \colon \R^\ndim \to [0, \infty)$ be a symmetric gauge function.
We say that $f$ is \emph{regular} if it is
differentiable away from $0$ and if there exist $c, C > 0$ such that the function $g = \frac{1}{2} f^2$ satisfies
\begin{equation} \label{eq:Hessian-estimate-g}
\frac{c}{2} |y - y_0|^2 \le g(y) - g(y_0) - dg(y_0)(y - y_0) \le \frac{C}{2} |y - y_0|^2
\end{equation}
for any $y, y_0 \in \R^\ndim$.
\end{definition}

If $f$ is not regular in this sense, then we can still construct
a family $(f_p)_{p > m}$ of regular symmetric gauge functions such that $f_q \le f_p$ when $q \le p$ and $f = \sup_{p > m} f_p$.
We will have weaker results in this case, but we can prove certain
statements by approximation.

We will derive further properties of symmetric gauge functions and, more importantly, of the corresponding functions $F = \frac{1}{2}\|\blank\|_f^2$ in Section \ref{sct:properties-of-F}.

We now discuss some tools that we require to remove the assumption that
$u \in C^1(M; N)$ in the results
described in the introduction. If we only know that $u \in C^{0, 1}(M; N)$, then the derivative $du$ is well-defined only almost everywhere on $M$ with respect to the measure $\vol$. It then makes no sense in general to say that $du$
or $\grad u$ satisfies an equation
almost everywhere with respect to another measure, such as $\|T\|$ for some $1$-current $T$, which may be singular relative to $\vol$.
Similarly, the $\so(n)$-valued $1$-form $\omega$, with components $\omega_{ij} = d\eta_i \cdot \eta_j$, is only well-defined almost everywhere with respect to $\vol$, and the above definition of $\partial^u T$ is meaningless in general.

It is convenient here to use the embedding $N \subseteq \R^\ndim$ and the vector fields $\eta_i = \bar{\eta}_i \circ u$.
This allows us to turn any $\R^n$-valued $1$-current $T$ on $M$ with finite mass
into an $\R^\ndim$-valued $1$-current $T^\dagger$, which has better properties in this respect. We define $T^\dagger$ by the formula
\[
T^\dagger(\psi) = \sum_{i = 1}^n \int_M \pairing{\psi}{\vec{T}_i \otimes \eta_i} \, d\|T\|
\]
for any $\R^\ndim$-valued $1$-form $\psi$ on $M$. (This means that $T^\dagger$ depends on $u$. The notation hides this fact, but as we will switch from $T$ to $T^\dagger$ only when $u$ is fixed, this will not cause any problems.)

Conversely, given an $\R^\ndim$-valued $1$-current $V$ on $M$ with finite mass, we have an $\R^n$-valued one $V^\ddagger$, defined by
\[
V^\ddagger(\sigma) = \sum_{i = 1}^n \int_M \pairing{\sigma_i}{\vec{V} \cdot \eta_i} \, d\|V\|.
\]
If $V = T^\dagger$, then $T = V^\ddagger$. The converse is true only if
$\vec{V}_j(x) \in T_{u(x)} N$ for $\|V\|$-almost every $x \in M$ and for
$j = 1, \dotsc, \ndim$.

We can then describe the $u$-boundary $\partial^u T$ in terms of $T^\dagger$. This is done in Section \ref{sct:vector-fields},
but for now, we only require an alternative characterisation of the equation $\partial^u T = 0$.
For this purpose, we use the following notation. Given a continuous map $\phi \colon M \to \R^\ndim$, we write $\phi^\perp(x)$ for the
orthogonal projection of $\phi(x)$ onto $T_{u(x)}^\perp N$. Then $\phi^\perp \colon M \to \R^\ndim$ is continuous.
Furthermore, the following holds true.

\begin{proposition} \label{prp:u-boundary}
Suppose that $u \in C^1(M; N)$, and let $T$ be an $\R^n$-valued $1$-current of finite mass. The equation
$\partial^u T = 0$ holds true if, and only if, there exists $C \ge 0$ such that
\begin{equation} \label{eq:vanishing-boundary-weak}
T^\dagger(d\phi) \le C \|\phi^\perp\|_{C^0(M)}
\end{equation}
for any $\phi \in C^\infty(M; \R^\ndim)$.
\end{proposition}

This observation suggests the following definition.

\begin{definition} \label{def:u-boundary}
Let $u \in C^{0, 1}(M; N)$, and let $T$ be an $\R^n$-valued $1$-current on $M$. We say that $T$ satisfies the equation $\partial^u T = 0$ \emph{weakly} if there exists a constant $C \ge 0$ such that \eqref{eq:vanishing-boundary-weak} holds true
for every $\phi \in C^\infty(M; \R^\ndim)$.
\end{definition}

To generalise the  equations \eqref{eq:T-represents-du-simplified} and \eqref{eq:T-almost-determines-du-simplified}
from Theorems \ref{thm:geodesic-simplified} and Theorem \ref{thm:weak-geodesic-simplified},
we rely on two different ideas, depending on the properties of $f$.
If $f$ is regular, then it suffices to regularise $u$ and prove convergence of the relevant quantities. Here we use the following concept.

\begin{definition}
A \emph{regular mollifier} is a family of maps $\mathcal{R}_\epsilon \colon C^{0, 1}(M; N) \to C^1(M; N)$ such that $\mathcal{R}_\epsilon u \to u$ uniformly
as $\epsilon \searrow 0$ and
\[
\limsup_{\epsilon \searrow 0} E_\infty(\mathcal{R}_\epsilon u) \le E_\infty(u)
\]
for any $u \in C^{0, 1}(M; N)$, and such that $d(\mathcal{R_\epsilon} u) \to du$
uniformly for any $u \in C^1(M; N)$.
\end{definition}

Regular mollifiers always exist. For example, they can be constructed by
\begin{itemize}
\item convolving with a standard mollification kernel in local coordinates,
\item projecting onto $N$ through the nearest point projection, and
\item gluing the pieces together with the help of a partition of unity.
\end{itemize}
As this process relies entirely on standard methods, the details are not given here.

If $f$ is not regular, then this tool is insufficient. In this situation,
we rely on the fact that a (conventional, not vector-valued) $1$-current $T$
can be decomposed into curves, provided that it is of finite mass and $\partial T$ is also represented by a Radon measure.
This idea goes back to the work of Smirnov \cite{Smirnov:93}, but the formulation of
Rod\'iguez-Arenas and Wengenroth \cite{Rodriguez-Arenas-Wengenroth:24} is more convenient for our purpose.

Let $\Gamma$ denote the space of all $1$-Lipschitz curves $\gamma \colon [0, 1] \to M$ (meaning that the Lipschitz constant is bounded by $1$), equipped with the uniform norm.
Given a $1$-current with the above properties, there exists a Radon measure $\nu$ on $\Gamma$
such that
\[
T(\sigma) = \int_\Gamma \int_0^1 \pairing{\sigma(\gamma(t))}{\dot{\gamma}(t)} \, dt \, d\nu(\gamma)
\]
for any $1$-form $\sigma$ on $M$.
(This result was proved for $\R^m$ instead of a manifold $M$ by Rodr\'igues-Arenas and Wengenroth. On a manifold, the measure $\nu$ can be constructed
by embedding $M$ in a Euclidean space $\R^{\bar{m}}$ and regarding $T$ as a
$1$-current on $\R^{\bar{m}}$ by virtue of the push-forward through the inclusion map.) Thus if we have a
continuously differentiable function $w \colon M \to \R$, then
\begin{equation} \label{eq:curve-decomposition}
T(dw) = \int_\Gamma \bigl(w(\gamma(1)) - w(\gamma(0))\bigr) \, d\nu.
\end{equation}
Without the continuous differentiability, the left-hand side loses its meaning, but the right-hand side
is well-defined for any continuous function. It can therefore be used to define weak versions of quantities such as $T(dw)$.

For our purpose, because we have a vector-valued current, we need to consider the
components separately. Similarly, we use the components of a map $u \in C^{0, 1}(M; N)$ in a formula similar to \eqref{eq:curve-decomposition}.

There will be some changes to the results from the introduction
because we work with a more general function
$F \colon \Lin(TM; \R^\ndim) \to [0, \infty)$ of the
form $F(X) = \frac{1}{2}\|X\|_f^2$ for some symmetric gauge function $f$,
rather than with $\frac{1}{2}|\blank|_2^2$ or $\frac{1}{2} |\blank|_\infty$.
The mass $\M$ is then no longer the appropriate quantity to consider. Instead, for an $\R^\ndim$-valued $1$-current $V$ on $M$, we define
\[
\M_F^\dagger(V) = \sup_{F(\psi) \le \frac{1}{2}} V(\psi)
\]
(where the supremum is taken over all smooth $\R^\ndim$-valued $1$-forms $\psi$
on $M$ that satisfy the inequality).
We also note here that $\vec{V}$ is a section of the vector bundle
$TM \otimes \R^\ndim$ over $M$. We freely identify $\R^\ndim$ with its dual,
and thus we may alternatively interpret $\vec{V}$ as a section of
$TM \otimes (\R^\ndim)^*$, or of $\Lin(\R^\ndim; TM)$, by which we mean
the vector bundle with fibre $\Lin(\R^\ndim; T_xM)$ at $x \in M$.
This makes the expression $F^*(\vec{V})$ meaningful, and provides
a different way to represent $\M_F^\dagger$. Assuming that $V$ has finite mass,
we have the formula
\begin{equation} \label{eq:F-mass}
\M_F^\dagger(V) = \int_M \sqrt{2 F^*(\vec{V})} \, d\|V\|.
\end{equation}
This was proved by Ignat and Moser \cite{Ignat-Moser:25} under somewhat different assumptions, but the same arguments apply here. For an $\R^n$-valued $1$-current $T$, we define
\[
\M_F(T) = \M_F^\dagger(T^\dagger).
\]
It is worth noting that $\M_F = \M$ when $F = \frac{1}{2} |\blank|_2^2$.

Finally, we may wish to push forward an $\R^n$-valued $1$-current with a smooth map $\Phi \colon M \to M$, and the procedure is the same as for conventional currents, applied component-wise. That is, for a vector-valued differential form $\sigma = (\sigma_1, \dotsc, \sigma_n)$, we define the pull-back $\Phi^\# \sigma = (\Phi^\# \sigma_1, \dotsc, \Phi^\# \sigma_n)$. An $\R^n$-valued $1$-current $T$ then has the push-forward $\Phi_\# T$ with
\[
\Phi_\# T(\sigma) = T(\Phi^\# \sigma).
\]
If $\Phi$ is bijective, we also define $\Phi_\# u = u \circ \Phi^{-1}$.
This notion is consistent with the equation $\partial^u T = 0$.

\begin{proposition} \label{prp:push-forward}
Suppose that $u \in C^{0, 1}(M; N)$ and that $T$ is an $\R^n$-valued $1$-current on $M$ of finite mass. Let $\Phi \colon M \to M$ be a smooth diffeomorphism. If $\partial^u T = 0$ weakly, then $\partial^{\Phi_\# u} (\Phi_\# T) = 0$ weakly.
\end{proposition}

\section{Main results} \label{sct:main-results}

In this section, we formulate the full version of the results alluded to in Theorems~\ref{thm:geodesic-simplified}--\ref{thm:weak-geodesic-simplified}. Throughout this section and the rest of the paper, we assume that
$f$ is a symmetric gauge function and $F \colon \Lin(TM; \R^\ndim) \to \R$ is the function defined by $F(A) = \frac{1}{2}\|A\|_f^2$.

A rough summary of this section is that the results from the introduction
still hold true for this choice of $F$ and under the assumption that
$u \in C^{0, 1}(M; N)$ only,
provided that all the quantities are interpreted in a suitably weak sense.
In particular, we have the following version of Theorem \ref{thm:geodesic-simplified}.

\begin{theorem} \label{thm:geodesic}
Suppose that $f$ is regular and that $u \in C^{0, 1}(M; N)$ is an $\infty$-harmonic map and $e_\infty = E_\infty(u) > 0$.
Then there exists an $\R^n$-valued $1$-current $T \neq 0$ on $M$ of finite mass such that
\begin{enumerate}
\item \label{itm:identity-for-du} the convergence
\begin{equation} \label{eq:represents-du}
\lim_{\epsilon \searrow 0} \int_M \biggl|\grad(\mathcal{R}_\epsilon u) - e_\infty \frac{\nabla F^*(\vec{T}^\dagger)}{\sqrt{F^*(\vec{T}^\dagger)}}\biggr|_2^2 \, d\|T\| = 0
\end{equation}
holds for any regular mollifier $(\mathcal{R}_\epsilon)_{\epsilon > 0}$;
\item \label{itm:boundary-vanishes} the equation $\partial^u T = 0$ holds weakly; and
\item for any smooth vector field $\xi$ on $M$,
\begin{equation} \label{eq:stationary}
\sum_{i = 1}^n \int_M \frac{\nabla F^*(\vec{T}^\dagger) : \nabla_{\vec{T}_i} \xi \otimes \eta_i}{\sqrt{2F^*(\vec{T}^\dagger)}} \, d\|T\| = 0.
\end{equation}
\end{enumerate}
\end{theorem}

We note that equation \eqref{eq:represents-du} can be thought of as a weaker version of
\begin{equation} \label{eq:pointwise-equation-for-du}
\grad u = e_\infty \frac{\nabla F^*(\vec{T}^\dagger)}{\sqrt{F^*(\vec{T}^\dagger)}},
\end{equation}
which replaces the equation $\grad u = \sqrt{2} e_\infty \sum_{i = 1}^n \vec{T}_i \otimes \eta_i$ from Theorem \ref{thm:geodesic-simplified} for more general choices of $f$. The quantity
$\nabla F^*(\vec{T}^\dagger)$ is interpreted as a section of $TM \otimes \R^\ndim$ here. The equation
$\partial^u T = 0$ is now satisfied only weakly, but has otherwise the same structure as in Theorem \ref{thm:geodesic-simplified}.
Equation \eqref{eq:stationary} means that $T$ is a critical point for the functional $\M_F$ (instead of $\M$). In fact, the following holds true.

\begin{proposition} \label{prp:first-variation-of-mass}
Let $T$ be an $\R^n$-valued $1$-current on $M$. Suppose that $(\Phi_t)_{t \in \R}$ is a smooth family of diffeomorphisms $\Phi_t \colon M \to M$. Let
$\xi = \dd{}{t}|_{t = 0} \Phi_t$.
Then
\[
\left.\frac{d}{dt}\right|_{t = 0} \M_F\bigl((\Phi_t)_\# T\bigr) = \sum_{i = 1}^n \int_M \frac{\nabla F^*(\vec{T}^\dagger) : \nabla_{\vec{T}_i} \xi \otimes \eta_i}{\sqrt{2F^*(\vec{T}^\dagger)}} \, d\|T\|.
\]
\end{proposition}

Even though the quantities $T^\dagger$ and $\M_F$ are defined in relation to some
map $u \in C^{0, 1}(M; N)$ in the preceding section, and no such map is
singled out in the proposition, the statement is still meaningful.
Because $F$ is orthogonally invariant, it is easy to see that the above quantities do in fact not depend on $u$, so we can choose it arbitrarily.

Theorem \ref{thm:equivalence-simplified} now takes the following form.

\begin{theorem} \label{thm:equivalence}
Suppose that $f$ is regular, and
let $u \in C^{0, 1}(M; N)$ and $e_\infty \ge E_\infty(u) > 0$. Suppose that
there exists an $\R^n$-valued $1$-current $T \neq 0$ on $M$ of finite mass such that
\[
\lim_{\epsilon \searrow 0} \int_M \biggl|\grad(\mathcal{R}_\epsilon u) - e_\infty \frac{\nabla F^*(\vec{T}^\dagger)}{\sqrt{F^*(\vec{T}^\dagger)}}\biggr|_2^2 \, d\|T\| = 0
\]
for some regular mollifier $(\mathcal{R}_\epsilon)_{\epsilon >0}$ and such that
$\partial^u T = 0$ weakly.
Then $u$ is $\infty$-harmonic and $e_\infty = E_\infty(u)$.
\end{theorem}

Once more, we have a result here that implies equivalence between the
condition that $u$ is $\infty$-harmonic, and the conjunction of
\ref{itm:identity-for-du} and \ref{itm:boundary-vanishes} from
Theorem~\ref{thm:geodesic}. Furthermore, if the conditions are satisfied for one
regular mollifier, then the same follows for any regular mollifier.

Finally, we consider symmetric gauge functions that are not necessarily regular.
The following result corresponds to Theorem \ref{thm:weak-geodesic-simplified}.

\begin{theorem} \label{thm:geodesic-weak}
Suppose that $u \in C^{0, 1}(M; N)$ is an $\infty$-harmonic map and $e_\infty = E_\infty(u) > 0$. Then there exists an $\R^n$-valued $1$-current $T \neq 0$ of finite mass, and there exist Radon measures $\nu_1, \dotsc, \nu_\ndim$ on $\Gamma$ such that
\begin{enumerate}
\item \label{itm:vanishing-boundary2} $\partial^u T = 0$ weakly;
\item \label{itm:bound-on-mass} $\M_F(T) \le \sqrt{2} e_\infty$;
\item for any smooth $\R^\ndim$-valued $1$-form $\psi = (\psi_1, \dotsc, \psi_\ndim)$,
\begin{equation} \label{eq:representation-by-curves}
T^\dagger(\psi) = \sum_{j = 1}^\ndim \int_\Gamma \int_0^1 \pairing{\psi_j(\gamma(t))}{\dot{\gamma}(t)} \, dt \, d\nu_j(\gamma);
\end{equation}
and
\item \label{itm:weak-representation-of-du} the equation
\begin{equation} \label{eq:weak-representation-of-du}
2e_\infty^2 = \sum_{j = 1}^\ndim \int_\Gamma \bigl(u_j(\gamma(1)) - u_j(\gamma(0))\bigr) \, d\nu_j(\gamma)
\end{equation}
holds true.
\end{enumerate}
\end{theorem}

Condition \ref{itm:vanishing-boundary2} is meanwhile familiar, of course.
Equation \eqref{eq:representation-by-curves} means that $T^\dagger$ is represented by the
measures $\nu_1, \dotsc, \nu_\ndim$ as described in the preceding section.
(We can think of a decomposition of the
components of $T^\dagger$ into curves.)
Condition \ref{itm:bound-on-mass} is perhaps not interesting on its own,
but together with \eqref{eq:weak-representation-of-du}, it gives rise to
a weak form of the equation
\[
\vec{T}^\dagger : \grad u = 2 e_\infty \sqrt{F^*(\vec{T}^\dagger)},
\]
which can itself be thought of as a variant of \eqref{eq:pointwise-equation-for-du}.
Indeed, the next proposition shows that the equation follows when we assume
some additional regularity.

\begin{proposition} \label{prp:consistency}
Suppose that $u \in C^1(M; N)$ and $e_\infty \ge E_\infty(u)$. If $T$ is an
$\R^n$-valued $1$-current $T \neq 0$ on $M$ of finite mass such that there exist measures
$\nu_1, \dotsc, \nu_\ndim$ on $\Gamma$
satisfying \eqref{eq:representation-by-curves} for every smooth $\R^\ndim$-valued
$1$-form $\psi$ and satisfying
\[
\sqrt{2}e_\infty \M_F(T) \le \sum_{j = 1}^\ndim \int_\Gamma \bigl(u_j(\gamma(1)) - u_j(\gamma(0))\bigr) \, d\nu_j(\gamma),
\]
then $e_\infty = E_\infty(u)$ and
\[
\vec{T}^\dagger : \grad u = 2 e_\infty \sqrt{F^*(\vec{T}^\dagger)}
\]
almost everywhere with respect to $\|T\|$. If $f$ is regular, then
\[
\grad u = e_\infty \frac{\nabla F^*(\vec{T}^\dagger)}{\sqrt{F^*(\vec{T}^\dagger)}}
\]
almost everywhere with respect to $\|T\|$.
\end{proposition}

It is then clear that Theorems \ref{thm:geodesic-simplified}--\ref{thm:weak-geodesic-simplified} follow from these results. More precisely, Theorem
\ref{thm:geodesic-simplified} is a consequence of
Theorem \ref{thm:geodesic} and Proposition \ref{prp:u-boundary};
Theorem \ref{thm:equivalence-simplified} follows from Theorem \ref{thm:equivalence} and Proposition \ref{prp:u-boundary}; and Theorem \ref{thm:weak-geodesic-simplified} is
a consequence of Theorem~\ref{thm:geodesic-weak} and Proposition \ref{prp:consistency}.

\section{Properties of $F$ and $F^*$} \label{sct:properties-of-F}

In preparation for the proofs of our main results, we derive some properties of
the function $F$ and its Legendre transform $F^*$. Above all, we require some estimates for
these functions and their derivatives.

Some of these estimates are easiest to formulate for some abstract finite-dimensional spaces $V$ and $W$. (Later on, we will choose $V = T_x M$ for some $x \in M$
and $W = \R^\ndim$.) Once more we assume that $f$ is a symmetric gauge function and $F(A) = \frac{1}{2} \|A\|_f^2$ for $A \in \Lin(V; W)$. We begin with the following simple observation.

\begin{lemma} \label{lem:F-gradient-estimate}
Let $C > 0$ be a number such that $f(y) \le C|y|$ for every $y \in \R^\ndim$.
If $F$ is differentiable, then
\[
|\nabla F(A)|_2 \le C^2 |A|_2
\]
for every $A \in \Lin(V; W)$.
\end{lemma}

\begin{proof}
We first estimate
\[
F(A) = \frac{1}{2} \bigl(f(s(A))\bigr)^2 \le \frac{C^2}{2} |s(A)|^2 = \frac{C^2}{2} |A|_2^2.
\]
Fix $A_0 \in \Lin(V; W)$ with $|A_0|_2^2 = 1$. By the convexity, we know that
\[
F(A_0) + \nabla F(A_0) : (A - A_0) \le F(A)
\]
for all $A \in \Lin(V; W)$. As $F$ is quadratically homogeneous, it further satisfies $\nabla F(A_0) : A_0 = 2 F(A_0)$. Hence
\[
\nabla F(A_0) : A \le F(A_0) + F(A) \le \frac{C^2}{2}(|A_0|_2^2 + |A|_2^2).
\]
Choosing
\[
A = \frac{\nabla F(A_0)}{|\nabla F(A_0)|_2},
\]
we see that
\[
|\nabla F(A_0)|_2 \le C^2.
\]
Because $\nabla F$ is homogeneous of degree $1$, the desired inequality now follows for all $A \in \Lin(V; W)$.
\end{proof}

If $f$ is not regular, then this result does not apply directly, because
$F$ may not be differentiable. But there always exist regular symmetric gauge
functions $f_p$ such that $f_q \le f_p$ when $q \le p$ and such that
$f = \sup_{p \in (m, \infty)} f_p$. As $f$ is a norm, there also
exists a constant $C$ such that
$f(y) \le C|y|$ for every $y \in \R^\ndim$. If we set $F_p = \frac{1}{2} \|\blank\|_{f_p}^2$, then $F_p$ is differentiable, as we will see in a moment.
Lemma \ref{lem:F-gradient-estimate} therefore
provides an estimate for $\nabla F_p$ that is uniform in $p$.

If we have a \emph{regular} symmetric gauge function, then the inequalities for $g$
from Definition \ref{def:regular} give rise to a similar property of $F$.

\begin{proposition} \label{prp:regular}
Suppose that $f$ is regular. Then there exist $c, C > 0$ such that $F$ is differentiable and satisfies
\begin{equation} \label{eq:Hessian-estimate-F}
\frac{c}{2} |A - A_0|_2^2 \le F(A) - F(A_0) - \nabla F(A_0) : (A - A_0) \le \frac{C}{2} |A - A_0|_2^2
\end{equation}
for any $A_0, A \in \Lin(V; W)$. Furthermore, its Legendre transform $F^*$ is also differentiable and satisfies
\begin{equation} \label{eq:Hessian-estimate-F*}
\frac{1}{2C} |B - B_0|_2^2 \le F^*(B) - F^*(B_0) - \nabla F^*(B_0) : (B - B_0) \le \frac{1}{2c} |B - B_0|_2^2
\end{equation}
for any $B_0, B \in \Lin(W; V)$.
\end{proposition}

\begin{proof}
Let $c > 0$ and $C > 0$ be the constants such that \eqref{eq:Hessian-estimate-g} holds true.
Choose $\tilde{c} \in (0, c)$ and $\Tilde{C} > C$. Then the functions
\[
g_1(y) = g(y) - \frac{\tilde{c}}{2} |y|^2 \quad \text{and} \quad g_2(y) = \frac{\tilde{C}}{2} |y|^2 - g(y)
\]
are convex and satisfy $g_1^{-1}(\{0\}) = g_2^{-1}(\{0\}) = \{0\}$. Hence the functions $f_1 = \sqrt{2 g_1}$ and $f_2 = \sqrt{2 g_2}$ are level-convex (i.e., their sublevel sets are convex). Since they are also positive $1$-homogeneous and symmetric, and since $f_1^{-1}(\{0\}) = f_2^{-1}(\{0\}) = \{0\}$, they are
norms on $\R^\ndim$. It follows that $f_1$ and $f_2$ are symmetric gauge functions.

Let $F_1(A) = g_1(s(A)) = F(A) - \frac{\tilde{c}}{2} |A|_2^2$ and $F_2(A) = g_2(s(A)) = \frac{\tilde{C}}{2}|A|_2^2 - F(A)$ for $A \in \Lin(V; W)$. Then these are convex functions as well. Thus for any fixed $A_0 \in \Lin(V; W)$, there exist $L_1, L_2 \in \Lin(V; W)$ such that
\[
F_1(A) \ge F_1(A_0) + L_1 : (A - A_0)
\]
and
\[
F_2(A) \ge F_2(A_0) + L_2 : (A - A_0)
\]
for all $A \in \Lin(V; W)$. Thus
\[
F(A) \ge F(A_0) + \frac{\tilde{c}}{2} |A - A_0|_2^2 + (L_1 + \tilde{c}A_0) : (A - A_0)
\]
and
\[
F(A) \le F(A_0) + \frac{\tilde{C}}{2} |A - A_0|_2^2 + (L_2 + \tilde{C}A_0) : (A - A_0).
\]
Both inequalities together are only possible when $L_1 + \tilde{c} A_0 = L_2 + \tilde{C} A_0$, and we conclude that $F$ is differentiable at $A_0$
with $\nabla F(A_0) = L_1 + \tilde{c} A_0 = L_2 + \tilde{C} A_0$.
We now obtain \eqref{eq:Hessian-estimate-F} by letting $\tilde{c} \to c$ and $\tilde{C} \to C$.
Since $F$ is thus \emph{strongly} convex and grows quadratically, the map $\nabla F \colon \Lin(V; W) \to \Lin(V; W)$ is a bijection.

We now consider the inequality
\[
F(A) \ge \frac{c}{2} |A - A_0|_2^2 + \nabla F(A_0) : (A - A_0) + F(A_0) 
\]
which follows from \eqref{eq:Hessian-estimate-F}, and take the Legendre transforms on both sides. This yields
\[
F^*(B) \le \frac{1}{2c} |B^* - \nabla F(A_0)|_2^2 + A_0 : (B^* - \nabla F(A_0)) + A_0 : \nabla F(A_0) - F(A_0)
\]
for any $B \in \Lin(W; V)$. Similarly, we see that
\[
F^*(B) \ge \frac{1}{2C} |B^* - \nabla F(A_0)|_2^2 + A_0 : (B^* - \nabla F(A_0)) + A_0 : \nabla F(A_0) - F(A_0).
\]
Given $B_0 \in \Lin(W; V)$, we choose the unique $A_0 \in \Lin(V; W)$ such that $B_0^* = \nabla F(A_0)$. Then these inequalities become
\begin{multline*}
\frac{1}{2C} |B - B_0|_2^2 + A_0^* : (B - B_0) + A_0^* : B_0 - F(A_0) \le F^*(B) \\
\le \frac{1}{2c} |B - B_0|_2^2 + A_0^* : (B - B_0) + A_0^* : B_0 - F(A_0).
\end{multline*}
Hence $F^*(B_0) = A_0^* : B_0 - F(A_0)$, and it also follows that
$F^*$ is differentiable at $B_0$ with $\nabla F^*(B_0) = A_0^*$. 
From this, we finally obtain \eqref{eq:Hessian-estimate-F*}.
\end{proof}

We also require the following facts.

\begin{lemma} \label{lem:images}
Suppose that $F$ is differentiable. Then for any $A \in \Lin(V; W)$,
\[
\im(\nabla F(A)) \subseteq \im(A).
\]
\end{lemma}

\begin{proof}
We want to show that for any $B \in \Lin(V; W)$, if
$\im(B) \subseteq (\im(A))^\perp$, then $\nabla F(A) : B = 0$.
This then clearly implies the statement of the lemma.

Let therefore $B$ have that property. Then it also follows that
$\im(A) \subseteq (\im(B))^\perp$, and hence
that $\im(A) \subseteq \ker(B^*)$ and $\im(B) \subseteq \ker(A^*)$.
Therefore, for any $t \in \R$,
\[
(A + tB)^* (A + tB) = A^* A + t^2 B^* B = (A - tB)^* (A - tB).
\]
Hence $s(A + tB) = s(A - tB)$ and $F(A + tB) = F(A - tB)$. Now we compute
\[
B : \nabla F(A) = \frac{d}{dt}\Bigr|_{t = 0} F(A + tB) = \frac{d}{dt}\Bigr|_{t = 0} F(A - tB) = - B : \nabla F(A).
\]
The claim thus follows.
\end{proof}

\begin{lemma} \label{lem:Cauchy-Schwarz-generalisation}
For any $A \in \Lin(V; W)$ and $B \in \Lin(W; V)$,
\[
\pairing{B}{A} \le 2\sqrt{F(A) F^*(B)}.
\]
If $f$ is regular, then equality holds true if, and only if, $A = 0$ or $B = 0$ or
\begin{equation} \label{eq:CS-equality}
\frac{B^*}{\sqrt{F^*(B)}} = \frac{\nabla F(A)}{\sqrt{F(A)}}.
\end{equation}
\end{lemma}

\begin{proof}
We may assume that $A \neq 0$ and $B \neq 0$. By the definition of the Legendre transform,
\begin{equation} \label{eq:Legendre-transform}
F^*(B') = \sup_{A' \in \Lin(V; W)} \bigl(\pairing{B'}{A'} - F(A')\bigr)
\end{equation}
for all $B' \in \Lin(W; V)$. Hence
\begin{equation} \label{eq:Young-for-F}
\pairing{B}{A} \le F(\alpha A) + F^*(\alpha^{-1} B)
\end{equation}
for any $\alpha \neq 0$. If we choose $\alpha = (F^*(B)/F(A))^{1/4}$, then we obtain the desired inequality.

It is readily checked that equality holds when \eqref{eq:CS-equality} is satisfied.

Conversely, if we have equality, then for the same choice of $\alpha$, we have equality in \eqref{eq:Young-for-F} as well.
Assuming that $f$ is regular, the above supremum in \eqref{eq:Legendre-transform} is attained at exactly one point, which is characterised by
the equation $(B')^* = \nabla F(A')$. This equation is thus satisfied for $B' = B/\alpha$ and $A' = \alpha A$,
and this yields \eqref{eq:CS-equality}.
\end{proof}

\section{Vector fields along $u$} \label{sct:vector-fields}

For a given map $u \colon M \to N$, a vector field along $u$ is a map
$X \colon M \to TN$ such that $X(x) \in T_{u(x)} N$ for all $x \in M$.
An example is (a suitable representative of) $\vec{T}^\dagger$ if
$T$ is an $\R^n$-valued $1$-current of finite mass on $M$.
In this section, we discuss some properties
of this and similar vector fields. In particular, we discuss some more of the
background to Definition~\ref{def:u-boundary} and give the proof of Proposition \ref{prp:u-boundary}. We also introduce a convenient tool that we will use for the proof of Theorem \ref{thm:geodesic},
which is a variant of the notion of a measure-function pair of Hutchinson \cite{Hutchinson:86}.

Recall the assumption that $N$ is isometrically embedded in a Euclidean space $\R^\ndim$, so that we can write $N \subseteq \R^\ndim$.
This makes it possible to define the Sobolev spaces
\[
W^{1, p}(M; N) = \set{u \in W^{1, p}(M; \R^\ndim)}{u(x) \in N \text{ for almost all } x \in M}
\]
for $p \in [1, \infty]$.
It also implies that there exists a number $R > 0$
such that on the $R$-neighbourhood $U_R(N)$ of $N$ in $\R^\ndim$,
there is a unique and smooth nearest point projection
$\pi_N \colon U_R(N) \to N$. Then its derivative $d\pi_N(y)$ at a point $y \in N$ is the orthogonal projection onto the tangent space $T_y N$.
We use the notation $\sff$ for the second fundamental form associated to the embedding $N \subseteq \R^\ndim$. Then the Hessian $d^2\pi_N(y)$ of the nearest point projection satisfies
\[
d^2\pi(y)(X, Y) = -\sff(y)(X, Y)
\]
for all $X, Y \in T_y N$. (These facts about the nearest point projection are proved, e.g., in a book by Simon \cite{Simon:96}.)

Given $u \colon M \to N$, we consider the pull-back vector bundle $u^\# TN$ over $M$, with the fibre $T_{u(x)} N$ at a point $x \in M$. (We generally use the symbol $\#$, rather than $*$, for push-forward and pull-back, because
we use $*$ for a different purpose, and because the notation is common for related concepts in geometric measure theory.)
If $u$ belongs to a Sobolev space $W^{1, p}(M; N)$ with $p > m$, then $u$
is continuous by the Sobolev embedding theorem, and the pull-back bundle is well-defined. A section of $u^\# TN$ is then a vector field along $u$ as described
above. We also consider other vector bundles related to $u^\# TN$, such as $T^* M \otimes u^\# TN$. For example, the derivative $du$ may be interpreted as a section
of $T^*M \otimes u^\# TN$, or alternatively, as a section of $\Lin(TM; u^\# TN)$,
the bundle over $M$ with fibre $\Lin(T_x M; T_{u(x)} N)$ at $x \in M$.

As previously mentioned, we assume that there exist smooth vector fields
$\bar{\eta}_1, \dotsc, \bar{\eta}_n \colon N \to \R^{\bar{n}}$ such that $(\bar{\eta}_1(y), \dotsc, \bar{\eta}_n(y))$ is an
orthonormal basis of $T_y N$ for every $y \in \N$.
We have a covariant derivative $\nabla^u$ acting on sections of $u^\# TN$,
which is induced by the Levi-Civita connection $D$ on $N$.
It can be expressed in terms of $\eta_i = \bar{\eta}_i \circ u$, $i = 1, \dotsc, n$, as follows: if $X$ is a section of $u^\# TN$, we can write it in the form $X = \sum_{i = 1}^n X^i \eta_i$ for certain functions $X^1, \dotsc, X^n$ on $M$.
Then
\[
\nabla_\xi^u X = \sum_{i = 1}^n \left(\xi(X^i) \eta_i + X^i D_{du(\xi)} \bar{\eta}_i(u)\right)
\]
for any tangent vector $\xi$ on $M$.
Alternatively, we can use the second fundamental form and write
\[
\nabla_\xi^u X = \xi(X) + \sff(u)(du(\xi), X),
\]
where $X$ is interpreted as a map into $\R^\ndim$ for the first term and the derivative $\xi(X)$ is then taken component-wise. We sometimes write
$\nabla_\xi X = \xi(X)$ to emphasise the fact that we differentiate in the direction
of $\xi$. The quantity $\nabla_\xi^u X$ also happens to be
the orthogonal projection of $\xi(X)$ onto $T_{u(x)} N$ at almost every point $x \in M$.

We now apply some of these observations to a vector field coming from
an $\R^n$-valued $1$-current $T$.

\begin{proof}[Proof of Proposition \ref{prp:u-boundary}]
We assume that $u \in C^1(M; N)$ here, and we consider an $\R^n$-valued
$1$-current $T$ of finite mass. We want to show that $\partial^u T = 0$ holds true
if, and only if, there exists a constant $C \ge 0$ such that
\begin{equation} \label{eq:weak-u-boundary-repeated}
T^\dagger(d\phi) \le C\|\phi^\perp\|_{C^0(M)}
\end{equation}
for all $\phi \in C^\infty(M; \R^\ndim)$.

Given $\phi \in C^\infty(M; \R^\ndim)$, let $\theta = (\phi \cdot \eta_1, \dotsc, \phi \cdot \eta_n)$. We compute
\[
T^\dagger(d\phi) = \sum_{i = 1}^n \int_M \nabla_{\vec{T}_i} \phi \cdot \eta_i \, d\|T\| = \int_M \pairing{d\theta}{\vec{T}} \, d\|T\| - \sum_{i = 1}^n \int_M \phi \cdot \nabla_{\vec{T}_i} \eta_i \, d\|T\|.
\]
Since $\eta_i = d\pi_N(u) \eta_i$, differentiation gives
\[
\nabla_\xi \eta_i = d\pi_N(u) \nabla_\xi \eta_i + d^2\pi_N(u)(du(\xi), \eta_i)
\]
for any vector field $\xi$ on $M$. Hence
\[
\begin{split}
\phi \cdot \nabla_\xi \eta_i & = \sum_{j = 1}^n (\phi \cdot \eta_j) (\eta_j \cdot \nabla_\xi \eta_i) - \phi \cdot \sff(u)(du(\xi), \eta_i) \\
& = \sum_{j = 1}^n \pairing{\omega_{ij}}{\xi} \theta_j - \phi \cdot \sff(u)(du(\xi), \eta_i).
\end{split}
\]
Therefore,
\[
\begin{split}
T^\dagger(d\phi) & = \int_M \pairing{d\theta - \omega \theta}{\vec{T}} \, d\|T\| + \sum_{i = 1}^n \int_M \phi \cdot \sff(u)(du(\vec{T}_i), \eta_i) \, d\|T\| \\
& = \partial^u T(\theta) + \int_M \phi \cdot \tr \sff(u)(du, \vec{T}^\dagger) \, d\|T\|.
\end{split}
\]
If $\partial^u T = 0$, then we obtain \eqref{eq:weak-u-boundary-repeated} from the observation that $\sff(u(x))(X, Y) \in T_{u(x)}^\perp N$ for any $X, Y \in T_{u(x)} N$.

Conversely, if there exists a constant such that \eqref{eq:weak-u-boundary-repeated} holds true
for every $\phi \in C^\infty(M; \R^\ndim)$, then we see by approximation that the same inequality holds for $\phi \in C^1(M; \R^\ndim)$. Given $\theta \in C^\infty(M; \R^n)$, we consider $\phi = \sum_{i = 1}^n \theta_i \eta_i$.
Since $\phi^\perp = 0$ for this choice, the above formula implies that
$\partial^u T(\theta) = 0$.
\end{proof}

In the proofs of our main results, and in particular in the proof of Theorem~\ref{thm:geodesic}, a section of the vector bundle $u^\# TN$, or similar, will
often appear together with a certain measure on $M$. We can then use the following extension of the concept of measure-function pairs of Hutchinson \cite{Hutchinson:86}, which has also been considered by the
author in a different context \cite{Moser:15.2}.

\begin{definition}
Let $\pi \colon E \to M$ be a vector bundle over $M$ with bundle metric
$\scp{\blank}{\blank}_E$ and corresponding bundle norm $|\blank|_E$.
A \emph{measure-section pair} over $E$ is a pair
$(\mu, h)$, where $\mu$ is a Radon measure on $M$ and $h \colon M \to E$
is a $\mu$-measurable function with $\pi(h(x)) = x$ for almost all $x \in M$
and such that $|h|_E \in L^1(M)$.
\end{definition}

If $E = u^\# TN$, then this amounts (in local coordinates, at least) to a measure-function
pair with values in $\R^\ndim$ in the sense of Hutchinson,
satisfying the additional condition $h(x) \in T_{u(x)} N$ almost everywhere.

With this notion comes a generalisation of some standard facts about strong and weak $L^2$-convergence.
In a situation where we have not just a sequence of functions (or sections of a
vector bundle), but where every member of the sequence has its own measure, we can work with the following concepts.

\begin{definition}
Let $(\mu_k, h_k)_{k \in \N}$ be a sequence of measure-section pairs over $E$.
We say that this sequence converges to the measure-section pair $(\mu_\infty, h_\infty)$ in the \emph{weak $L^2$-sense} if
\[
\limsup_{k \to \infty} \int_M |h_k|_E^2 \, d\mu_k < \infty
\]
and
\[
\lim_{k \to \infty} \int_ M \bigl(\chi + \scp{h_k}{\phi}_E\bigr) \, d\mu_k = \int_ M \bigl(\chi + \scp{h_\infty}{\phi}_E\bigr) \, d\mu_\infty
\]
for all $\chi \in C^0(M)$ and all continuous sections $\phi$ of $E$.
We say that the convergence is \emph{strong} if in addition,
\[
\lim_{k \to \infty} \int_M \Phi \circ h_k \, d\mu_k = \int_M \Phi \circ h_\infty \, d\mu_\infty
\]
for all continuous functions $\Phi \colon E \to \R$ such that there exists a
constant $C$ satisfying $|\Phi(X)| \le C(|X|_E^2 + 1)$ for all $X \in E$.
\end{definition}

The formulation of strong convergence is different in Hutchinson's paper \cite{Hutchinson:86},
but it is not difficult to see that the above is equivalent.
The following facts were proved by Hutchinson \cite[Theorem 4.4.2]{Hutchinson:86}
for trivial bundles over $M$. The same arguments apply,
\emph{mutatis mutandis}, in the general case as well.

\begin{proposition} \label{prp:Hutchinson}
Let $(\mu_k, h_k)_{k \in \N}$ be a sequence of measure-section pairs over $E$.
\begin{enumerate}
\item If
\[
\limsup_{k \to \infty} \int_M (|h_k|_E^2 + 1) \, d\mu_k < \infty,
\]
then some subsequence converges in the weak $L^2$-sense.
\item If $(\mu_k, h_k)$ converges to $(\mu_\infty, h_\infty)$ in the weak
$L^2$-sense, then
\[
\int_M |h_\infty|_E^2 \, d\mu_\infty \le \liminf_{k \to \infty} \int_M |h_k|_E^2 \, d\mu_k.
\]
\item If $(\mu_k, h_k)$ converges to $(\mu_\infty, h_\infty)$ in the weak
$L^2$-sense and
\[
\int_M |h_\infty|_E^2 \, d\mu_\infty = \limsup_{k \to \infty} \int_M |h_k|_E^2 \, d\mu_k,
\]
then the convergence is strong.
\end{enumerate}
\end{proposition}

\section{Approximation by $p$-harmonic maps} \label{sct:construction}

In this section, we will prove Proposition \ref{prp:energy-discrepancy} and
Theorem \ref{thm:geodesic-weak}. At the same time, we will lay the
foundations for the proof of  Theorem \ref{thm:geodesic}.

There exists a constant $c > 0$ and there exists a family $(f_p)_{p > m}$ of regular symmetric gauge functions
that is non-decreasing in $p$, such that $f = \sup_{p > m} f_p$ and $\sqrt{c} |y| \le f_p(y)$ for all
$p > m$ and all $y \in \R^\ndim$. Define
$F_p(X) = \frac{1}{2} \|X\|_{f_p}^2$. (If $f$ is already regular, then we choose $f_p = f$ for every $p > m$
and choose $c$ as in Definition \ref{def:regular}.)

Suppose that $u \in C^{0,1}(M; N)$ is an $\infty$-harmonic map,
and define $e_\infty = E_\infty(u)$. We assume that $e_\infty > 0$.
Using the embedding $N \subseteq \R^\ndim$, we can formulate the condition from Definition \ref{def:infinity-harmonic}
as follows:
\[
\liminf_{v \to u} \frac{E_\infty(v) - E_\infty(u)}{\|u - v\|_{C^0(M)}} \ge 0.
\]
Thus if we define
\[
\omega_0(r) = \sup_{\|u - v\|_{C^0(M)} \le r} \bigl(E_\infty(u) - E_\infty(v)\bigr), \quad r \ge 0,
\]
and $\omega_1(r) = \max\{\omega_0(r), r^2\}$, then $\omega_1$ has a vanishing right-hand derivative
at $0$. Now choose a function $\chi \in C_0^\infty((-1, 0))$ with $\int_{-1}^0 \chi \, ds = 1$ and
set
\[
\omega(r) = \frac{1}{r} \int_r^{2r} \chi\left(1 - \frac{s}{r}\right) \omega_1(s) \, ds, \quad r > 0,
\]
and $\omega(0) = 0$. Because $\omega_1$ is non-decreasing, we find that $\omega(r) \ge \omega_1(r) \ge \omega_0(r)$
for every $r \ge 0$. Thus
\begin{equation} \label{eq:u-is-infinity-harmonic}
E_\infty(u) \le E_\infty(v) + \omega\bigl(\|u - v\|_{C^0(M)}\bigr)
\end{equation}
for any $v \in C^{0, 1}(M; N)$. Furthermore, $\omega \in C^1([0, \infty))$
with $\omega'(0) = 0$, and $\omega(r) \ge r^2$ for all $r \ge 0$.

We now consider the functionals
\[
E_\infty^0(v) = E_\infty(v) + 2\omega\bigl(\|u - v\|_{C^0(M)}\bigr).
\]
Moreover, for $p \in (m, \infty)$, we define
\[
E_p(v) = \left(\fint_M \bigl(F_p(dv)\bigr)^{p/2} \, d\vol\right)^{1/p}
\]
and
\[
E_p^0(v) = E_p(v) + 2\omega\left(\left(\fint_M |v - u|^p \, d\vol\right)^{1/p}\right).
\]
For every $p \in (m, \infty)$, there exists a minimiser $u_p \in W^{1, p}(M; N)$ of $E_p^0$, which can be constructed with the direct method.
Set $e_p = E_p(u_p)$.
As $p > m$, each $u_p$ is continuous by the Sobolev embedding theorem. Using
well-known methods from the theory of $p$-harmonic maps \cite{Hardt-Lin:87, Duzaar-Mingione:04}, we can then also show that $u_p \in C^1(M; N)$.

\begin{lemma} \label{lem:convergence-to-u}
As $p \to \infty$, the maps $u_p$ converge to $u$ weakly in $W^{1, q}(M; \R^\ndim)$ for every $q < \infty$. Furthermore,
\[
e_\infty = \lim_{p \to \infty} e_p.
\]
\end{lemma}

\begin{proof}
Because $u_p$ minimises $E_p^0$, and by H\"older's inequality, we know that
\[
E_q(u_p) \le E_p(u_p) \le E_p^0(u_p) \le E_p^0(u) = E_p(u) \le E_\infty(u)
\]
whenever $q \le p$. Hence $(u_p)_{p \in (q, \infty)}$ is bounded in $W^{1, q}(M; \R^\ndim)$, and there exist a sequence $p_k \to \infty$ and a map
\[
u_\infty \in \bigcap_{q \in (m, \infty)} W^{1, q}(M; N)
\]
such that we have the simultaneous weak convergence $u_{p_k} \rightharpoonup u_\infty$
in $W^{1, q}(M; \R^\ndim)$ for every $q < \infty$.
At the same time, we have $u_{p_k} \to u_\infty$ uniformly.

We can use Egorov's theorem to show that $E_\infty(v) = \lim_{q \to \infty} E_q(v)$
for any $v \in C^{0, 1}(M; N)$. Moreover, it follows from the convexity of
$F_p$ that the functional $E_p^0$ is lower
semicontinuous with respect to weak convergence in $W^{1, p}(M; \R^\ndim)$.
Hence
\[
\begin{split}
E_\infty^0(u_\infty) & = \lim_{q \to \infty} E_q^0(u_\infty) \\
& \le \liminf_{q \to \infty} \liminf_{k \to \infty} E_q^0(u_{p_k}) \\
& \le \liminf_{k \to \infty} E_{p_k}^0(u_{p_k}) \\
& \le E_\infty(u).
\end{split}
\]
Taking \eqref{eq:u-is-infinity-harmonic} into account, we find that
\[
\omega\bigl(\|u - u_\infty\|_{C^0(M)}\bigr) \le 0.
\]
Thus $u = u_\infty$.

Since this means that we have a unique limit independent of the sequence $(p_k)_{k \in \N}$, we conclude that $u_p \rightharpoonup u$ weakly in
$W^{1, q}(M; \R^\ndim)$ for every $q < \infty$. Furthermore, we obtain the uniform convergence
$u_p \to u$. Hence
\[
\lim_{p \to \infty} \left(E_p^0(u_p) - e_p\right) = 0.
\]
H\"older's inequality implies that
\[
E_q^0(u_q) \le E_q^0(u_p) \le E_p^0(u_p)
\]
when $q \le p$, and it follows that $\lim_{p \to \infty} e_p$ exists.
The above inequalities then also show that
\[
e_\infty = E_\infty(u) \le \lim_{p \to \infty} e_p \le E_\infty(u).
\]
This concludes the proof.
\end{proof}

As a consequence of this lemma (and of the assumption that $e_\infty > 0$), we know that $e_p > 0$ when $p$ is sufficiently large. If we set
\[
\lambda_p = \left(\fint_M |u_p - u|^p \, d\vol\right)^{1/p},
\]
then we have the Euler-Lagrange equations
\begin{multline*}
e_p^{1 - p} \biggl(d^*\Bigl(\bigl(F_p(du_p)\bigr)^{\frac{p - 2}{2}} \nabla F_p(du_p)\Bigr) - \bigl(F_p(du_p)\bigr)^{\frac{p - 2}{2}} \tr \sff(u_p)(\nabla F_p(du_p), du_p)\biggr) \\
= -4\lambda_p^{1 - p} \omega'(\lambda_p) |u_p - u|^{p - 2} d\pi_N(u_p)(u_p - u),
\end{multline*}
where $d^*$ is the $L^2$-adjoint of the exterior derivative.
This equation is derived by calculating $\frac{d}{dt}|_{t = 0} E_p^0(\pi_N(u_p + t\phi))$ for $\phi \in C^\infty(M; \R^\ndim)$. The arguments are essentially
the same as for harmonic maps (explained, e.g., in a book by Simon \cite{Simon:96}).  If instead, we consider the variations
$u_p \circ \Phi_t$ for
a smooth family of diffeomorphisms $\Phi_t \colon M \to M$ with $\Phi_0(x) = x$
for all $x \in M$ and with $\xi = \dd{}{t}|_{t = 0} \Phi_t$, then we obtain a different condition:
\[
\begin{split}
0 & = e_p^{1 - p} \int_M \bigl(F_p(du_p)\bigr)^{\frac{p - 2}{2}} \left(\nabla F_p(du_p) : du_p(\nabla \xi) - \frac{2}{p} F_p(du_p) \div \xi\right) \, d\vol \\
& \quad + 4\lambda_p^{1 - p} \omega'(\lambda_p) \int_M |u_p - u|^{p - 2} (u_p - u) \cdot du_p(\xi) \, d\vol.
\end{split}
\]
Here the expression $du_p(\nabla \xi)$ stands for the section of
$T^*M \otimes u_p^\# TN$ given, in local coordinates on $M$, by
\[
\sum_{\alpha = 1}^m dx^\alpha \otimes du_p(\nabla_{\partial/\partial x^\alpha} \xi).
\]
Once more, the computations are similar to those for harmonic maps
\cite{Price:83, Grosse-Brauckmann:92}.
(If $\lambda_p = 0$, then the terms involving $\lambda_p^{1 - p} \omega'(\lambda_p)$ should be
replaced by $0$ in both equations.)

It is now convenient to introduce the measures
\[
\mu_p = e_p^{2 - p} \bigl(F_p(du_p)\bigr)^{\frac{p - 2}{2}} \frac{\vol}{\vol(M)}
\]
on $M$. Then these equations can be expressed in the form
\begin{multline} \label{eq:Euler-Lagrange-L}
\int_M \bigl(\nabla F_p(du_p) : d\phi - \phi \cdot \tr \sff(u_p)(\nabla F_p(du_p), du_p)\bigr) \, d\mu_p \\
= - 4 e_p \lambda_p^{1 - p} \omega'(\lambda_p) \fint_M |u_p - u|^{p - 2} \phi \cdot d\pi_N(u_p)(u_p - u) \, d\vol
\end{multline}
for all $\phi \in C^\infty(M; \R^\ndim)$, and
\begin{multline} \label{eq:inner-variations-L}
\int_M \left(\nabla F_p(du_p) : du_p(\nabla \xi) - \frac{2}{p} F_p(du_p) \div \xi\right) \, d\mu_p \\
= - 4e_p \lambda_p^{1 - p} \omega'(\lambda_p) \fint_M |u_p - u|^{p - 2} (u_p - u) \cdot du_p(\xi) \, d\vol
\end{multline}
for all smooth tangent vector fields $\xi$ on $M$.
Note also that
\[
\mu_p(M) = e_p^{2 - p} \fint_M \bigl(F_p(du_p)\bigr)^{\frac{p - 2}{2}} \, d\vol \le 1
\]
by H\"older's inequality, and
\[
\int_M F_p(du_p) \, d\mu_p = e_p^2.
\]
Because $f_p(y) \ge \sqrt{c}|y|$ for $y \in \R^\ndim$, we conclude that
\[
\int_M |du_p|_2^2 \, d\mu_p \le \frac{2e_p^2}{c}.
\]

\begin{proof}[Proof of Proposition \ref{prp:energy-discrepancy}]
Let $v \in C^{0, 1}(M; N)$. We want to show that
\[
E_\infty(v) \ge e_\infty\bigl(1 - C\|u - v\|_{C^0(M)}^2\bigr)
\]
for a constant $C$ that depends neither on $u$ nor on $v$.

Fix a regular mollifier $(\mathcal{R}_\epsilon)_{\epsilon > 0}$ and set $v_\epsilon = \mathcal{R}_\epsilon v$.
Because of the convexity of $F_p$ and the Euler-Lagrange equation \eqref{eq:Euler-Lagrange-L},
we have the inequality
\begin{multline*}
\int_M F_p(dv_\epsilon) \, d\mu_p \\
\begin{aligned}
& \ge \int_M \bigl(F_p(du_p) + \nabla F_p(du_p) : (dv_\epsilon - du_p)\bigr) \, d\mu_p \\
& = \int_M F_p(du_p) \, d\mu_p + \int_M (v_\epsilon - u_p) \cdot \tr \sff(u_p)(\nabla F_p(du_p), du_p) \, d\mu_p \\
& \quad - 4 e_p \lambda_p^{1 - p} \omega'(\lambda_p) \fint_M |u_p - u|^{p - 2} (v_\epsilon - u_p) \cdot d\pi_N(u_p)(u_p - u) \, d\vol
\end{aligned}
\end{multline*}
for every $p \in (m, \infty)$. Note that there exists a constant $C_1$, depending only on $N$,
such that for any $y, z \in N$ and for any $Y, Z \in T_yN$,
\[
(z - y) \cdot \sff(y)(Y, Z) \le C_1 |y - z|^2 |Y| |Z|,
\]
because $\sff(y)(Y, Z)$ takes values in in the normal space $T_y^\perp N$ and
$z - y$ is tangential to $N$ at $y$ to leading order. We further know that
$du_p(x) \in T_x^*M \otimes T_{u(x)} N$ and $\nabla F_p(du_p(x)) \in T_x^*M \otimes T_{u(x)} N$,
by Lemma \ref{lem:images}, for every $x \in M$. Lemma~\ref{lem:F-gradient-estimate} implies that
there exists a constant $C_2$, depending only on $F$, such that $|\nabla F_p(du_p)| \le C_2|du_p|$
for all $p$. Hence
\[
\int_M F_p(dv_\epsilon) \, d\mu_p \ge e_p^2 - 2c^{-1} C_1C_2 e_p^2 \|v_\epsilon - u_p\|_{C^0(M)}^2 - 4\diam(N) e_p \omega'(\lambda_p).
\]
Since $F \ge F_p$, it follows that
\[
\sup_{x \in M} F(dv_\epsilon(x)) \ge e_p^2 - 2c^{-1} C_1C_2 e_p^2 \|v_\epsilon - u_p\|_{C^0(M)}^2 - 4\diam(N) e_p \omega'(\lambda_p).
\]
Letting $p \to \infty$, we obtain
\[
\sup_{x \in M} F(dv_\epsilon(x)) \ge e_\infty^2 - 2c^{-1} C_1C_2 e_\infty^2 \|v_\epsilon - u\|_{C^0(M)}^2.
\]
By the properties of the regular mollifier $\mathcal{R}_\epsilon$, it follows that
\[
\begin{split}
E_\infty(v) & \ge e_\infty \sqrt{1 - 2c^{-1} C_1C_2 \|v_\epsilon - u_p\|_{C^0(M)}^2} \\
& \ge e_\infty \left(1 - c^{-1} C_1C_2 \|v_\epsilon - u_p\|_{C^0(M)}^2\right).
\end{split}
\]
This is the desired inequality.
\end{proof}

\begin{proof}[Proof of Theorem \ref{thm:geodesic-weak}]
Consider the $\R^\ndim$-valued $1$-currents $V_p$ on $M$
such that
\[
V_p(\psi) = \int_M \nabla F_p(du_p) : \psi \, d\mu_p
\]
for any $\R^\ndim$-valued $1$-form $\psi$ on $M$. By \eqref{eq:Euler-Lagrange-L}
and Lemma \ref{lem:F-gradient-estimate}, we have a constant $C_1$, depending only on $c$, $e_\infty$, $f$, $N$, and $\omega$, such that
\[
|V_p(d\phi)| \le C_1 \|\phi\|_{C^0(M)}
\]
for any $\phi \in C^\infty(M; \R^\ndim)$. Hence each component of $V$ can be regarded as a
$1$-current on $M$ of finite mass, the boundary of which is represented by a Radon measure. By the results of Rodr\'iguez-Arenas and Wengenroth \cite{Rodriguez-Arenas-Wengenroth:24} (adapted to the manifold setting), there exist Radon measures $\nu_{pj}$ on $\Gamma$, for $j = 1, \dotsc, \ndim$, such that
\begin{equation} \label{eq:representation-by-curves-p}
V_p(\psi) = \sum_{j = 1}^\ndim \int_\Gamma \int_0^1 \pairing{\psi_j(\gamma(t))}{\dot{\gamma}(t)} \, dt \, d\nu_{pj}
\end{equation}
for any $\ndim$-valued $1$-form $\psi$ on $M$. Examining the proof \cite[Section 4]{Rodriguez-Arenas-Wengenroth:24},
we further see that
\[
\nu_{pj}(\Gamma) \le C_2
\]
for a constant $C_2$ depending only on $c$, $e_\infty$, $f$, $N$, and $\omega$.
We write $u_p = (u_{p1}, \dotsc, u_{p\ndim})$ now.
If we choose $\psi = du_p$ in \eqref{eq:representation-by-curves-p}, then
\begin{equation} \label{eq:representation-by-curves2}
\begin{split}
e_p^2 &= \int_M F_p(du_p) \, d\mu_p \\
& = \frac{1}{2} \int_M \nabla F_p(u_p) : du_p \, d\mu_p \\
& = \frac{1}{2} V_p(du_p) \\
& = \frac{1}{2} \sum_{j = 1}^\ndim \int_\Gamma \int_0^1 \pairing{du_{pj}(\gamma(t))}{\dot{\gamma}(t)} \, dt \, d\nu_{pj} \\
& = \frac{1}{2} \sum_{j = 1}^\ndim \int_\Gamma \bigl(u_{pj}(\gamma(1)) - u_{pj}(\gamma(0))\bigr) \, d\nu_{pj}.
\end{split}
\end{equation}

The quantity $(\nabla F_p(du_p))^*$ is a section of $\Lin(\R^\ndim; TM)$, but
can alternatively be regarded as a section of $TM \otimes \R^\ndim$.
With this interpretation, we have $\vec{V}_p \|V_p\| = (\nabla F_p(du_p))^* \mu_p$.
Note furthermore that $F_p^*((\nabla F_p(du_p))^*) = F_p(du_p)$, and therefore,
\[
\begin{split}
\M_{F_p}^\dagger(V_p) & = \int_M \sqrt{2F_p^*\bigl((\nabla F_p(du_p))^*\bigr)} \, d\mu_p \\
& \le \left(2\mu_p(M) \int_M F_p(du_p) \, d\mu_p\right)^{\frac{1}{2}} \\
& \le \sqrt{2} e_p.
\end{split}
\]
Moreover, since $F_p \le F$ for all $p \in (m, \infty)$, and hence $F_p^* \ge F^*$, it follows that $\M_F^\dagger(V_p) \le \sqrt{2} e_p$.

Because we have constants $C_3, C_4$ such that
\[
\|V_p\|(M) = \int_M |\nabla F_p(du_p)|_2 \, d\mu_p \le C_3 \int_M |du_p|_2 \, d\mu_p \le C_4 e_p,
\]
which remains bounded as $p \to \infty$, we can find a sequence $p_k \to \infty$ such that $V_{p_k} \stackrel{*}{\rightharpoonup} V_\infty$, in the weak*-sense, for
some $\R^\ndim$-valued $1$-current $V_\infty$ of finite mass. At the same time, we may assume (after passing to another subsequence if necessary) that $\nu_{p_k j} \stackrel{*}{\rightharpoonup} \nu_{\infty j}$, for $j = 1, \dotsc, \ndim$, in the weak*-topology of $C^0(\Gamma)$. It was shown by Rodr\'iguez-Arenas and Wengenroth \cite{Rodriguez-Arenas-Wengenroth:24} that the functionals
\[
\gamma \mapsto \int_0^1 \pairing{\psi_j(\gamma(t))}{\dot{\gamma}(t)} \, dt
\]
are continuous on $\Gamma$. Hence it follows from \eqref{eq:representation-by-curves-p} that
\[
V_\infty(\psi) = \sum_{j = 1}^\ndim \int_\Gamma \int_0^1 \pairing{\psi_j(\gamma(t))}{\dot{\gamma}(t)} \, dt \, d\nu_{\infty j}.
\]
As $\M_F^\dagger$ is lower semicontinuous with respect to the above convergence, it follows that $\M_F^\dagger(V_\infty) \le \sqrt{2} e_\infty$.
Because we know that $u_{p_k} \to u$ uniformly, identity \eqref{eq:representation-by-curves2} implies that
\[
e_\infty^2 = \sum_{j = 1}^\ndim \int_\Gamma \bigl(u_j(\gamma(1)) - u_j(\gamma(0))\bigr) \, d\nu_{\infty j}.
\]
If we can find an $\R^n$-valued current $T$ on $M$ such that $V_\infty = T^\dagger$, then this gives the properties \ref{itm:bound-on-mass}--\ref{itm:weak-representation-of-du} in Theorem \ref{thm:geodesic-weak}.

Moreover, approximating $u$ by smooth maps, we conclude that there exists
$v \in C^\infty(M; \R^\ndim)$ such that
\[
\begin{split}
V_\infty(dv) & = \sum_{j = 1}^\ndim \int_\Gamma \int_0^1 \pairing{dv_j(\gamma(t))}{\dot{\gamma}(t)} \, dt \, \, d\nu_{\infty j} \\
& = \sum_{j = 1}^\ndim \int_\Gamma \int_0^1 \bigl(v_j(\gamma(1)) - v_j(\gamma(0))\bigr) \, dt \, \, d\nu_{\infty j} \neq 0.
\end{split}
\]
Hence $V_\infty \neq 0$.

Recall that
\[
\lambda_p = \left(\fint_M |u_p - u|^p \, d\vol\right)^{1/p},
\]
and we know that $\lambda_p \to 0$ as $p \to \infty$ by Lemma \ref{lem:convergence-to-u}. For $\phi \in C^\infty(M; \R^\ndim)$, if $\phi_p^\perp$ denotes the projection of $\phi$
onto $T_{u_p}^\perp N$ and $\phi^\perp$ denotes the projection onto $T_u^\perp N$,
then we have the uniform convergence $\phi_p^\perp \to \phi^\perp$. Recall
that we have a constant $c > 0$ such that $\frac{c}{2} |du_p|^2 \le F_p(du_p)$ for all $p > m$.
Furthermore, by Lemma \ref{lem:F-gradient-estimate}, we know that $|\nabla F_p(du_p)|_2 \le C_5 |du_p|_2$ for some constant $C_5$. As
\[
\bigl|\phi \cdot \tr \sff(u_p)(\nabla F_p(du_p), du_p)\bigr| \le C_6 |\nabla F_p(du_p)|_2 |du_p|_2 |\phi_p^\perp|
\]
for some constant $C_6$ depending only on $N$, the Euler-Lagrange equation \eqref{eq:Euler-Lagrange-L} implies that
\begin{equation} \label{eq:u-boundary-vanishes}
\begin{split}
V_\infty(d\phi) & = \lim_{k \to \infty} \int_M \nabla F_{p_k} : d\phi \, d\mu_{p_k} \\
& \le 2c^{-1} C_5 C_6 \limsup_{k \to \infty} \int_M F_{p_k}(du_{p_k}) |\phi_{p_k}^\perp| \, d\mu_{p_k} \\
& \le 2c^{-1} C_5 C_6 e_\infty^2 \|\phi^\perp\|_{C^0(M)}.
\end{split}
\end{equation}

Now consider an $\R^\ndim$-valued, continuous $1$-form $\psi$ on $M$ such that $\psi(x) \in T_x^* M \otimes T_{u(x)}^\perp N$ at every $x \in M$. Since $\nabla F_p(du_p(x)) \in T_x^* M \otimes T_{u(x)} N$ everywhere by Lemma \ref {lem:images},
it follows that
\[
V_\infty(\psi) = \lim_{k \to \infty} \int_M \nabla F_{p_k}(du_{p_k}) : \psi \, d\mu_{p_k} = 0.
\]
This implies that $\vec{V}_\infty(x) \in T_x M \otimes T_{u(x)} N$ for $\|V_\infty\|$-almost every $x \in M$. Hence if we define the $n$-valued $1$-current $T = V^\ddagger_\infty$, then $V_\infty = T^\dagger$. 
Then \eqref{eq:u-boundary-vanishes} means that
$\partial^u T = 0$ weakly.
Thus we have verified the property \ref{itm:vanishing-boundary2} in Theorem \ref{thm:geodesic-weak} as well.
\end{proof}

We conclude this section by remarking that all of the above observations also apply
in the situation of Theorem \ref{thm:geodesic} for $F_p = F$.

\section{Strong convergence} \label{sct:strong-convergence}

In this section, we complete the proof of Theorem \ref{thm:geodesic}.
We continue to use the notation and some of the results from the preceding section, but we now assume that $f$ is regular and set
$F_p = F$ for every $p > m$. We now interpret $(\mu_p, du_p)$ as
measure-section pairs over $T^*M \otimes \R^\ndim$. We have the inequality
\[
\int_M |du_p|_2^2 \, d\mu_p \le \frac{2}{c} \int_M F(du_p) \, d\mu_p = \frac{2e_p^2}{c}
\]
by Proposition \ref{prp:regular}, and we recall that
\[
\mu_p(M) = e_p^{2 - p} \fint_M \bigl(F(du)\bigr)^{\frac{p - 2}{2}} \, d\vol \le 1
\]
by H\"older's inequality. According to Proposition \ref{prp:Hutchinson},
there exists a sequence $p_k \to \infty$ such that $(\mu_{p_k}, du_{p_k})$ converges in the weak $L^2$-sense
to a measure-section pair $(\mu_\infty, Z_\infty)$.

We need to improve this convergence. We use the same strategy as in the paper by
Katzourakis and Moser \cite{Katzourakis-Moser:25} for this purpose. Some of the underlying ideas go back
to a paper by Evans and Yu \cite{Evans-Yu:05}, including the following lemma.

\begin{lemma} \label{lem:energy-estimate}
For any $p \in (m, \infty)$ and any $\beta \in (0, 1)$,
\[
\int_M F(du_p) \, d\mu_p \ge \beta^2 e_p^2 \mu_p(M) - \beta^p e_p^2.
\]
\end{lemma}

\begin{proof}
Let $\Lambda_p = \set{x \in M}{F(du_p(x)) \le \beta^2 e_p^2}$. Then
\[
\mu_p(\Lambda_p) = \frac{e_p^{2 - p}}{\vol(M)} \int_{\Lambda_p} \bigl(F(du_p)\bigr)^{\frac{p - 2}{2}} \, d\vol \le \beta^{p - 2}.
\]
Hence
\[
\begin{split}
\int_M F(du_p) \, d\mu_p & \ge \int_{M \setminus \Lambda_p} F(du_p) \, d\mu_p \\
& \ge \beta^2 e_p^2 \mu_p(M \setminus \Lambda_p) \\
& = \beta^2 e_p^2 \bigl(\mu_p(M) - \mu_p(\Lambda_p)\bigr) \\
& \ge \beta^2 e_p^2 \mu_p(M) - \beta^p e_p^2,
\end{split}
\]
which is the desired inequality.
\end{proof}

We can now prove the following.

\begin{proposition} \label{prp:strong-convergence}
The convergence $(\mu_{p_k}, du_{p_k}) \to (\mu_\infty, Z_\infty)$ is strong in the $L^2$- sense. The limit satisfies $F(Z_\infty) = e_\infty^2$ at $\mu_\infty$-almost every point. Furthermore, for any regular mollifier $(\mathcal{R}_\epsilon)_{\epsilon > 0}$,
\[
\lim_{\epsilon \searrow 0} \int_M |d(\mathcal{R}_\epsilon u) - Z_\infty|_2^2 \, d\mu_\infty = 0.
\]
\end{proposition}

\begin{proof}
Let $v_\epsilon = \mathcal{R}_\epsilon u$. Inequality \eqref{eq:Hessian-estimate-F} and the Euler-Lagrange equation \eqref{eq:Euler-Lagrange-L} imply that
\begin{multline*}
\frac{c}{2} \int_M |dv_\epsilon - du_p|_2^2 \, d\mu_p \\
\begin{aligned}
& \le \int_M \bigl(F(dv_\epsilon) - F(du_p)\bigr) \, d\mu_p - \int_M \nabla F(du_p) : (dv_\epsilon - du_p) \, d\mu_p \\
& = \int_M \bigl(F(dv_\epsilon) - F(du_p)\bigr) \, d\mu_p - \int_M (v_\epsilon - u_p) \cdot \tr \sff(u_p)(\nabla F(du_p), du_p) \, d\mu_p \\
& \quad + 4 e_p \lambda_p^{1 - p} \omega'(\lambda_p) \fint_M |u_p - u|^{p - 2} (v_\epsilon - u_p) \cdot d\pi_N(u_p)(u_p - u) \, d\vol.
\end{aligned}
\end{multline*}
With the same estimates as in the proof of Proposition \ref{prp:energy-discrepancy},
we derive the inequality
\begin{multline*}
\frac{c}{2} \int_M |dv_\epsilon - du_p|_2^2 \, d\mu_p \le \int_M \bigl(F(dv_\epsilon) - F(du_p)\bigr) \, d\mu_p \\
+ C_1 e_p^2 \|v_\epsilon - u_p\|_{C^0(M)}^2 + 4\diam(N) e_p \omega'(\lambda_p)
\end{multline*}
for some constant $C_1$ that depends only on $N$, $c$, and $C$ (the constants from \eqref{eq:Hessian-estimate-F}).
Combining this with Lemma \ref{lem:energy-estimate}, we obtain
\begin{multline*}
\frac{c}{2} \int_M |dv_\epsilon - du_p|_2^2 \, d\mu_p \le \int_M F(dv_\epsilon) \, d\mu_p + \beta^p e_p^2 - \beta^2 e_p^2 \mu_p(M) \\
+ C_1 e_p^2 \|v_\epsilon - u_p\|_{C^0(M)}^2 + 4 \diam(N) e_p \omega'(\lambda_p).
\end{multline*}

We now restrict this to $p_k$ and let $k \to \infty$. This gives
\begin{multline*}
\frac{c}{2} \limsup_{k \to \infty} \int_M |dv_\epsilon - du_{p_k}|_2^2 \, d\mu_{p_k} \le \int_M F(dv_\epsilon) \, d\mu_\infty - \beta^2 e_\infty^2 \mu_\infty(M) \\
+ C_1 e_\infty^2 \|v_\epsilon - u\|_{C^0(M)}^2.
\end{multline*}
By the properties of the regular mollifier $\mathcal{R}_\epsilon$, there exists a function $\gamma \colon (0, 1] \to (0, \infty)$ with $\lim_{\epsilon \searrow 0} \gamma(\epsilon) = 0$, such that
\[
\esssup_M F(dv_\epsilon) \le \esssup_M F(du) + \gamma(\epsilon) e_\infty^2 \le e_\infty^2 (1 + \gamma(\epsilon))
\]
and
\[
|v_\epsilon - u|^2 \le \gamma(\epsilon)
\]
everywhere on $M$. Hence
\[
\frac{c}{2} \limsup_{k \to \infty} \int_M |dv_\epsilon - du_{p_k}|_2^2 \, d\mu_{p_k} \le  (1 - \beta^2 + \gamma(\epsilon)) e_\infty^2 \mu_\infty(M) + C_1 e_\infty^2 \gamma(\epsilon).
\]
Since this holds true for any $\beta \in (0, 1)$, it follows that
\begin{equation} \label{eq:strong-convergence-epsilon}
\lim_{\epsilon \searrow 0} \limsup_{k \to \infty} \int_M |dv_\epsilon - du_{p_k}|_2^2 \, d\mu_{p_k} = 0.
\end{equation}
Furthermore, we compute
\begin{multline*}
\limsup_{k \to \infty} \int_M |dv_\epsilon - du_{p_k}|_2^2 \, d\mu_{p_k} \\
\begin{aligned}
& = \limsup_{k \to \infty} \left(\int_M |dv_\epsilon|_2^2 \, d\mu_{p_k} - 2 \int_M dv_\epsilon : du_{p_k} \, d\mu_{p_k} + \int_M |du_{p_k}|_2^2 \, d\mu_{p_k}\right) \\
& = \int_M |dv_\epsilon|_2^2 \, d\mu_\infty - 2 \int_M dv_\epsilon : Z_\infty \, d\mu_\infty + \limsup_{k \to \infty} \int_M |du_{p_k}|_2^2 \, d\mu_{p_k} \\
& = \int_M |dv_\epsilon - Z_\infty|_2^2 \, d\mu_\infty - \int_M |Z_\infty|_2^2 \, d\mu_\infty + \limsup_{k \to \infty} \int_M |du_{p_k}|_2^2 \, d\mu_{p_k}.
\end{aligned}
\end{multline*}
Proposition \ref{prp:Hutchinson} and \eqref{eq:strong-convergence-epsilon} therefore imply that
\begin{equation} \label{eq:L2-convergence}
\lim_{\epsilon \searrow 0} \int_M |dv_\epsilon - Z_\infty|_2^2 \, d\mu_\infty = 0
\end{equation}
and that
\begin{equation} \label{eq:strong-convergence}
\limsup_{k \to \infty} \int_M |du_{p_k}|_2^2 \, d\mu_{p_k} \le \int_M |Z_\infty|_2^2 \, d\mu_\infty.
\end{equation}
Equation \eqref{eq:L2-convergence} amounts to the last statement of the proposition, whereas \eqref{eq:strong-convergence} gives the strong $L^2$-convergence $(\mu_{p_k}, du_{p_k}) \to (\mu_\infty, Z_\infty)$
by Proposition \ref{prp:Hutchinson}.

One of the consequences of this strong convergence is that
\begin{equation} \label{eq:integral-F(Z)}
\int_M F(Z_\infty) \, d\mu_\infty = \lim_{k \to \infty} \int_M F(du_{p_k}) \, d\mu_{p_k} = e_\infty^2.
\end{equation}
Furthermore, we know that $F(du) \le e_\infty^2$ almost everywhere (with respect to the measure $\vol$). By the properties of the regular mollifier $\mathcal{R}_\epsilon$, this means that
\[
\limsup_{\epsilon \searrow 0} \sup_M F(dv_\epsilon) \le e_\infty^2.
\]
Because of the $L^2$-convergence \eqref{eq:L2-convergence}, we obtain
\[
\|F(Z_\infty)\|_{L^\infty(\mu_\infty)} \le e_\infty^2.
\]
Now we combine this information with \eqref{eq:integral-F(Z)}, and we conclude that $F(Z_\infty) = e_\infty^2$ almost everywhere (with respect to $\mu_\infty$).
\end{proof}

\begin{proof}[Proof of Theorem \ref{thm:geodesic}]
One consequence of Proposition \ref{prp:strong-convergence} is that we can pass to the limit
in the Euler-Lagrange equation \eqref{eq:Euler-Lagrange-L} and in \eqref{eq:inner-variations-L}. We conclude that
\begin{equation} \label{eq:Euler-Lagrange-limit}
\int_M \bigl(\nabla F(Z_\infty) : d\phi - \phi \cdot \tr \sff(u)(\nabla F(Z_\infty), Z_\infty)\bigr) \, d\mu_\infty = 0
\end{equation}
for every $\phi \in C^\infty(M; \R^\ndim)$, and
\begin{equation} \label{eq:inner-variations-limit}
\int_M \nabla F(Z_\infty) : Z_\infty(\nabla \xi) \, d\mu_\infty = 0
\end{equation}
for every smooth tangent vector field $\xi$ on $M$.

It also follows from Proposition \ref{prp:strong-convergence} that 
the measure-section pair $(\mu_\infty, Z_\infty)$ is a different representation
of the $\R^\ndim$-valued current $V_\infty$ from the preceding section.
Indeed, we obtain
\[
V_\infty(\psi) = \int_M \nabla F(Z_\infty) : \psi \, d\mu_\infty
\]
for any $\R^n$-valued $1$-form $\psi$.
Hence
\[
\vec{V}_\infty= \frac{(\nabla F(Z_\infty))^*}{|\nabla F(Z_\infty)|_2}
\] 
and $\|V_\infty\| = |\nabla F(Z_\infty)|_2 \mu_\infty$.

By the standard properties of the Legendre transform and the fact that $F$ is homogeneous of degree $2$, and by Proposition \ref{prp:strong-convergence},
\[
F^*(\vec{V}_\infty) = \frac{1}{2} \vec{V}_\infty : \nabla F^*(\vec{V}_\infty) = \frac{\nabla F(Z_\infty) : Z_\infty}{2|\nabla F(Z_\infty)|_2^2} = \frac{F(Z_\infty)}{|\nabla F(Z_\infty)|_2^2} = \frac{e_\infty^2}{|\nabla F(Z_\infty)|_2^2}
\]
almost everywhere with respect to $\mu_\infty$. Hence
\[
|\nabla F(Z_\infty)|_2 = \frac{e_\infty}{\sqrt{F^*(\vec{V}_\infty)}}
\]
and
\[
Z_\infty^* = e_\infty \frac{\nabla F^*(\vec{V}_\infty)}{\sqrt{F^*(\vec{V}_\infty)}}.
\]
The last statement of Proposition \ref{prp:strong-convergence} can then be expressed in the form
\[
\lim_{\epsilon \searrow 0} \int_M \biggl|\grad(\mathcal{R}_\epsilon u) - e_\infty \frac{\nabla F^*(\vec{V}_\infty)}{\sqrt{F^*(\vec{V}_\infty)}}\biggr|_2^2 \sqrt{F^*(\vec{V}_\infty)} \, d\|V\| = 0.
\]
But because $|\vec{V}_\infty|_2 \equiv 1$, the function $\sqrt{F^*(\vec{V}_\infty)}$ is uniformly bounded below. Hence for $T = V^\ddagger_\infty$, equation
\eqref{eq:represents-du} from
Theorem \ref{thm:geodesic} also follows.

With similar calculations, equation \eqref{eq:inner-variations-limit} implies \eqref{eq:stationary}. We already know from the preceding section that
$T \neq 0$ and that $\partial^u T = 0$ weakly (and the latter also follows from \eqref{eq:Euler-Lagrange-limit}).
This completes the proof of Theorem \ref{thm:geodesic}.
\end{proof}

\section{Proof of Theorem \ref{thm:equivalence}} \label{sct:equivalence}

We still assume that $f$ is regular.
For a given map $u \in C^{0, 1}(M; N)$ such that $e_\infty \ge E_\infty(u) > 0$, we now fix a specific regular mollifier $(\mathcal{R}_\epsilon)_{\epsilon >0}$ and
set $v_\epsilon = \mathcal{R}_\epsilon u$. We then assume that there exists an $\R^n$-valued $1$-current of finite mass such that
\begin{equation} \label{eq:convergence-gradient}
\lim_{\epsilon \searrow 0} \int_M \biggl|\grad v_\epsilon - e_\infty \frac{\nabla F^*(\vec{T}^\dagger)}{\sqrt{F^*(\vec{T}^\dagger)}}\biggr|_2^2 \, d\|T\| = 0
\end{equation}
and that there exists $C \ge 0$ such that
\begin{equation} \label{eq:vanishing-boundary-repeated}
T^\dagger(d\phi) \le C \|\phi^\perp\|_{C^0(M)}
\end{equation}
for any $\phi \in C^\infty(M; \R^\ndim)$. We want to show that $e_\infty = E_\infty(u)$ and that
$u$ is $\infty$-harmonic.

We consider the section
\[
Z = e_\infty \frac{(\nabla F^*(\vec{T}^\dagger))^*}{\sqrt{F^*(\vec{T}^\dagger)}}
\]
of $T^*M \otimes \R^\ndim$ and the measure
\[
\mu = \sqrt{F^*(\vec{T}^\dagger)} \|T\|
\]
on $M$. We compute
\[
F(Z) = \frac{1}{2} Z : \nabla F(Z) = \frac{e_\infty^2}{2 F^*(\vec{T}^\dagger)} \nabla F^*(\vec{T}^\dagger) : \vec{T}^\dagger = e_\infty^2
\]
almost everywhere with respect to $\mu$. Then \eqref{eq:vanishing-boundary-repeated} becomes
\[
\int_M \nabla F(Z) : d\phi \, d\mu \le Ce_\infty \|\phi^\perp\|_{C^0(M)}.
\]

Let $w \in C^{0, 1}(M; N)$ be another map. For $\delta > 0$, set $w_\delta = \mathcal{R}_\delta w$.
Using \eqref{eq:Hessian-estimate-F}, we now find that
\[
\begin{split}
\frac{c}{2} \int_M |dw_\delta - Z|^2 \, d\mu & \le \int_M \bigl(F(dw_\delta) - F(Z)\bigr) \, d\mu - \int_M \nabla F(Z) : (dw_\delta - Z) \, d\mu \\
& = \int_M \bigl(F(dw_\delta) - F(Z)\bigr) \, d\mu - \int_M \nabla F(Z) : (dv_\epsilon - Z) \, d\mu \\
& \quad + \int_M \nabla F(Z) : (dv_\epsilon - dw_\delta) \, d\mu\\
& \le \int_M \bigl(F(dw_\delta) - F(Z)\bigr) \, d\mu - \int_M \nabla F(Z) : (dv_\epsilon - Z) \, d\mu \\
& \quad + Ce_\infty \|(w_\delta - v_\epsilon)^\perp\|_{C^0(M)}
\end{split}
\]
for any $\epsilon > 0$. As we let $\epsilon \searrow 0$, we have the convergence
$dv_\epsilon \to Z$ in $L^2(\mu; T^*M \otimes \R^\ndim)$ by \eqref{eq:convergence-gradient} (and by the fact that $F^*(\vec{T}^\dagger)$ is bounded). Moreover, $v_\epsilon \to u$ uniformly. Hence
\[
\frac{c}{2} \int_M |dw_\delta - Z|^2 \, d\mu \le \int_M \bigl(F(dw_\delta) - F(Z)\bigr) \, d\mu + Ce_\infty \|(w_\delta - u)^\perp\|_{C^0(M)}.
\]
There exists a constant $C_1$, depending only on $N$, such that
\[
|(w_\delta - u)^\perp| \le C_1 |w_\delta - u|^2.
\]
It follows that
\[
\begin{split}
\int_M F(dw_\delta) \, d\mu & \ge \int_M F(Z) \, d\mu - C_2 \|w_\delta - u\|_{C^0(M)}^2 \\
& = e_\infty^2 \mu(M) - C_2 \|w_\delta - u\|_{C^0(M)}^2
\end{split}
\]
for some constant $C_2$.
By the properties of the regular mollifier $\mathcal{R}_\epsilon$, this means that
\[
E_\infty(w) \ge e_\infty \left(1 - \frac{C_2}{e_\infty^2 \mu(M)} \|w - u\|_{C^0(M)}^2\right)^{1/2}.
\]
Choosing $w = u$, we conclude in particular that $E_\infty(u) = e_\infty$.

By the concavity of the function $t \mapsto \sqrt{1 - t}$, we also obtain the inequality
\[
E_\infty(w) \ge e_\infty - \frac{C_2}{2e_\infty \mu(M)} \|w - u\|_{C^0(M)}^2
\]
for any other $w \in C^{0, 1}(M; N)$. Thus $u$ is an $\infty$-harmonic map.

\section{Deformations of a current} \label{sct:deformations}

In this section we prove Proposition \ref{prp:push-forward} and Proposition \ref{prp:first-variation-of-mass}.
This means studying how a vector-valued $1$-current
behaves under a deformation given by a diffeomorphism of $M$. This is not significantly different
from conventional ($\R$-valued) currents by itself, but we do need to pay some attention to how these deformations
interact with the equation $\partial^u T = 0$ (in the weak form) and with the functional $\M_F$.

\begin{proof}[Proof of Proposition \ref{prp:push-forward}]
Let $\Phi \colon M \to M$ be a smooth diffeomorphism.
We consider an $\R^n$-valued $1$-current $T$ of finite mass such that $\partial^u T = 0$ weakly,
and we wish to show that $\partial^{\Phi_\# u} (\Phi_\# T) = 0$ weakly.

Define $S = \Phi_\# T$ and $v = \Phi_\# u = u \circ \Phi^{- 1}$. We also define
$\zeta_i = \eta_i \circ \Phi^{-1} = \bar{\eta}_i \circ v$ for $i = 1, \dotsc, n$.

Because we work with two different maps $u$ and $v$ here, we do not use the notation $T^\dagger$, but write $V$ for the $\R^\ndim$-valued current such that
\[
V(\psi) = \sum_{i = 1}^n \int_M \pairing{\psi}{\vec{T}_i \otimes \eta_i} \, d\|T\|
\]
for all $\R^\ndim$-valued $1$-forms $\psi$ on $M$. Similarly, we define $W$ by
the formula
\[
W(\psi) = \sum_{i = 1}^n \int_M \pairing{\psi}{\vec{S}_i \otimes \zeta_i} \, d\|S\|.
\]
For a map $\phi \colon M \to \R^\ndim$, let $\phi_u^\perp$ and $\phi_v^\perp$ denote the pointwise orthogonal projections onto $T_{u(x)}^\perp N$ and
$T_{v(x)}^\perp N$, respectively.

By the definition of $S = \Phi_\# T$,
for any $\R^n$-valued $1$-form $\sigma = (\sigma_1, \dotsc, \sigma_n)$, we can write
\[
\begin{split}
S(\sigma) & = T(\Phi^\# \sigma) \\
& = \int_M \pairing{\Phi^\# \sigma}{\vec{T}} \, d\|T\| \\
& = \int_M \pairing{\sigma(\Phi(x))}{d\Phi(x) \vec{T}(x)} \, d\|T\|(x) \\
& = \int_M \pairing{\sigma(y)}{d\Phi(\Phi^{-1}(y)) \vec{T}(\Phi^{-1}(y))} \, d(\Phi_\# \|T\|)(y),
\end{split}
\]
where $\Phi_\# \|T\|$ denotes the push-forward of $\|T\|$ in the measure-theoretic sense,
and where the linear map $d\Phi(x) \colon T_x M \to T_{\Phi(x)} M$ is applied component-wise in the expression $d\Phi(x) \vec{T}(x)$.
Hence
\[
\vec{S}(y) = \frac{d\Phi(\Phi^{-1}(y)) \vec{T}(\Phi^{-1}(y))}{|d\Phi(\Phi^{-1}(y)) \vec{T}(\Phi^{-1}(y))|_2}
\]
and
\[
\|S\| = |d\Phi(\Phi^{-1}(y)) \vec{T}(\Phi^{-1}(y))|_2 \Phi_\# \|T\|.
\]

Now let $\phi \in C^\infty(M; \R^\ndim)$. Then
\[
\begin{split}
W(d\phi) & = \sum_{i = 1}^n \int_M \zeta_i \cdot \nabla_{\vec{S}_i} \phi \, d\|S\| \\
& = \sum_{i = 1}^n \int_M \eta_i(\Phi^{-1}(y)) \cdot \nabla_{d\Phi(\Phi^{-1}(y)) \vec{T}_i(\Phi^{-1}(y))} \phi(y) \, d(\Phi_\# \|T\|)(y) \\
& = \sum_{i = 1}^n \int_M \eta_i(x) \cdot \nabla_{d\Phi(x) \vec{T}_i(x)} \phi(\Phi(x)) \, d\|T\|(x) \\
& = \sum_{i = 1}^n \int_M \eta_i \cdot \nabla_{\vec{T}_i} (\phi \circ \Phi) \, d\|T\| \\
& = V(d(\phi \circ \Phi)).
\end{split}
\]

If $\partial^u T = 0$ weakly, then there exists therefore a constant $C \ge 0$ such that
\[
|W(d\phi)| \le C\|(\phi \circ \Phi)_u^\perp\|_{C^0(M)}.
\]
But clearly $(\phi \circ \Phi)_u^\perp = \phi_v^\perp \circ \Phi$. Hence
\[
|W(d\phi)| \le C\|\phi_v^\perp\|_{C^0(M)},
\]
which means that $\partial^v S = 0$ weakly.
\end{proof}

\begin{proof}[Proof of Proposition \ref{prp:first-variation-of-mass}]
In this proof, we consider a smooth family of diffeomorphisms $\Phi_t \colon M \to M$ with $\Phi_0(x) = x$ for all $x \in M$. Given an $\R^n$-valued $1$-current
$T$ of finite mass, we need to compute
\[
\frac{d}{dt}\Bigr|_{t = 0} \M_F((\Phi_t)_\# T).
\]
As the Legendre transform $F^*$ appears in the formula \eqref{eq:F-mass} for $\M_F$, we begin with
some observations about this function.

For a section $X$ of $\Lin(\R^\ndim; TM)$ (or, equivalently, of $TM \otimes \R^\ndim$), recall that $F^*(X) = \frac{1}{2} \bigl(f^*(s(X^*))\bigr)^2$ for the dual of some symmetric gauge function $f$. In particular, it depends only on $X^*X$.
In local coordinates on $M$ (say, in an open set $\Omega \subseteq M$), suppose that
$X = \sum_{\alpha = 1}^m \dd{}{x^\alpha} \otimes X^\alpha$, where $X^\alpha \colon \Omega \to \R^\ndim$ for $\alpha = 1, \dotsc, m$. Let
$(g_{\alpha\beta})_{\alpha, \beta = 1, \dotsc, m}$ denote the Riemannian metric in the local coordinates. Then we can write
\[
F^*(X) = G(X^*X) = G\left(\sum_{\alpha, \beta = 1}^m g_{\alpha\beta} X^\alpha \otimes X^\beta\right)
\]
for some symmetric function $G \colon \R^{\ndim \times \ndim} \to \R$ (where we identify $\Lin(\R^\ndim; \R^\ndim)$ with $\R^{\ndim \times \ndim}$, and
`symmetric' means that $G(\Xi) = G(\Xi^\transpose)$ for all $\Xi \in \R^{\ndim \times \ndim}$).
Hence
\[
F^*\bigl(d\Phi_t X\bigr) =  G\left(\sum_{\alpha, \beta, \gamma, \delta = 1}^m (g_{\alpha\beta} \circ \Phi_t) \dd{\Phi_t^\alpha}{x^\gamma} \dd{\Phi_t^\beta}{x^\delta} X^\gamma \otimes X^\delta\right),
\]
and
\[
\dd{}{t}\Bigr|_{t = 0} F^*\bigl(d\Phi_t X\bigr) = \sum_{\alpha, \beta, \gamma = 1}^m \left(2g_{\alpha\beta} \dd{\xi^\beta}{x^\gamma} + \dd{g_{\alpha\gamma}}{x^\beta} \xi^\beta\right) dG(X^*X) X^\alpha \otimes X^\gamma.
\]
Using the fact that the metric tensor is parallel, expressed through
the formula
\[
\dd{g_{\alpha\gamma}}{x^\beta} = \sum_{\delta = 1}^m (g_{\delta\alpha} \Gamma_{\gamma\beta}^\delta + g_{\gamma\delta} \Gamma_{\beta\alpha}^\delta)
\]
(where $\Gamma_{\gamma\beta}^\delta$ denote the Christoffel symbols), we obtain
\[
\begin{split}
\dd{}{t}\Bigr|_{t = 0} F^*\bigl(d\Phi_t X\bigr) & = 2 \sum_{\alpha, \beta, \gamma = 1}^m g_{\alpha\beta} (\nabla_{\partial/\partial x^\gamma}\xi)^\beta dG(X^*X) X^\alpha \otimes X^\gamma \\
& = \sum_{j = 1}^\ndim \nabla F^*(X) : \nabla_X \xi,
\end{split}
\]
where $\nabla_X \xi$ is shorthand notation for $\sum_{\alpha = 1}^m \nabla_{\partial/\partial x^\alpha} \xi \otimes X^\alpha$.

For notational convenience, we fix an arbitrary map $u \in C^{0, 1}(M; N)$ and
write $V_t = ((\Phi_t)_\# T)^\dagger$. Then
\[
\M_F((\Phi_t)_\# T) = \int_M \sqrt{2F^*(\vec{V}_t)} \, d\|(\Phi_t)_\# T\|.
\]
We note that by the orthogonal
invariance of $F^*$, this quantity is independent of the choice of $u$.

With the same calculations as in the proof of Proposition \ref{prp:u-boundary}, we see that
\[
\int_\Omega \sqrt{2F^*(\vec{V}_t)} \, d\|(\Phi_t)_\# T\| = \int_\Omega \sqrt{2F^*\bigl(d\Phi_t \vec{V}_0\bigr)} \, d\|T\|.
\]
Hence we compute
\[
\frac{d}{dt}\Bigr|_{t = 0} \M_F((\Phi_t)_\# T) = \sum_{j = 1}^\ndim \int_M \frac{\nabla F^*(\vec{T}^\dagger) : \nabla_{T^\dagger} \xi}{\sqrt{2 F^*(\vec{T}^\dagger)}} \, d\|T\|,
\]
which is another representation of the formula in Proposition \ref{prp:first-variation-of-mass}.
\end{proof}

\section{Proof of Proposition \ref{prp:consistency}} \label{sct:consistency}

In this section, we prove Proposition \ref{prp:consistency}. We therefore consider
$u \in C^1(M; N)$ again. We further assume that $e_\infty \ge E_\infty(u)$ and that there is an $\R^n$-valued $1$-current $T \neq 0$ with finite mass,
as well as certain measures $\nu_1, \dotsc, \nu_\ndim$ on $\Gamma$, such that
\begin{equation} \label{eq:representation-by-curves3}
T^\dagger(\psi) = \sum_{j = 1}^\ndim \int_\Gamma \int_0^1 \pairing{\psi_j(\gamma(t))}{\dot{\gamma}(t)} \, dt \, d\nu_j(\gamma)
\end{equation}
for all $\R^\ndim$-valued $1$-forms $\psi$ on $M$, and such that
\begin{equation} \label{eq:weak-representation-of-du2}
\sqrt{2} e_\infty \M_F(T) \le \sum_{j = 1}^\ndim \int_\Gamma \bigl(u_j(\gamma(1)) - u_j(\gamma(0))\bigr) \, d\nu_j(\gamma).
\end{equation}

Because $T$ has finite mass and $du$ is continuous, equation \eqref{eq:representation-by-curves3} applies to $\psi = du$. Thus we compute
\[
\begin{split}
\sqrt{2} e_\infty \M_F(T) & \le \sum_{j = 1}^\ndim \int_\Gamma \bigl(u_j(\gamma(1)) - u_j(\gamma(0))\bigr) \, d\nu_j \\
& = \sum_{j = 1}^\ndim \int_\Gamma \int_0^1 \pairing{du(\gamma(t))}{\dot{\gamma}(t)} \, dt \, d\nu_j \\
& = T^\dagger(du).
\end{split}
\]
We also know that
\[
\M_F(T) = \M_F^\dagger(T^\dagger) = \int_M \sqrt{2F^*(\vec{T}^\dagger)} \, d\|T\|.
\]
With the help of Lemma \ref{lem:Cauchy-Schwarz-generalisation}, we obtain
the inequalities
\[
\begin{split}
T^\dagger(du) & = \int_M \pairing{\vec{T}^\dagger}{du} \, d\|T\| \\
& \le 2\|F(du)\|_{C^0(M)}^{1/2} \int_M \sqrt{F^*(\vec{T}^\dagger)} \, d\|T\| \\
& \le \sqrt{2} e_\infty \M(T).
\end{split}
\]
Hence we have in fact equality everywhere, and this implies that
\[
e_\infty^2 = F(du)
\]
and
\[
\vec{T}^\dagger : \grad u = \pairing{\vec{T}^\dagger}{du} = 2\sqrt{F(du) F^*(\vec{T}^\dagger)} = 2e_\infty \sqrt{F^*(\vec{T}^\dagger)}
\]
at $\|T\|$-almost every point. From the first of these identities,
we obtain $e_\infty = E_\infty(u)$, as $u \in C^1(M; N)$.

If $f$ is regular, then Lemma \ref{lem:Cauchy-Schwarz-generalisation} further implies that
\[
\grad u = \sqrt{\frac{F(du)}{F^*(\vec{T}^\dagger)}} \nabla F^*(\vec{T}^\dagger) = e_\infty \frac{\nabla F^*(\vec{T}^\dagger)}{\sqrt{F^*(\vec{T}^\dagger)}}
\]
almost everywhere with respect to $\|T\|$.

\begin{acknowledgement}
This work was supported by the Engineering and Physical Sciences Research Council [grant number EP/X017206/1].
\end{acknowledgement}

\def\cprime{$'$}
\providecommand{\bysame}{\leavevmode\hbox to3em{\hrulefill}\thinspace}
\providecommand{\MR}{\relax\ifhmode\unskip\space\fi MR }
\providecommand{\MRhref}[2]{%
  \href{http://www.ams.org/mathscinet-getitem?mr=#1}{#2}
}
\providecommand{\href}[2]{#2}


\begin{thebibliography}{10}

\bibitem{Aronsson:65}
G.~Aronsson, \emph{Minimization problems for the functional {${\rm
  sup}_{x}\,F(x,\,f(x),\,f^{\prime} (x))$}}, Ark. Mat. \textbf{6} (1965),
  33--53.

\bibitem{Aronsson:66}
\bysame, \emph{Minimization problems for the functional {${\rm sup}_{x}\, F(x,
  f(x),f\sp\prime (x))$}. {II}}, Ark. Mat. \textbf{6} (1966), 409--431.

\bibitem{Aronsson:67}
\bysame, \emph{Extension of functions satisfying {L}ipschitz conditions}, Ark.
  Mat. \textbf{6} (1967), 551--561.

\bibitem{Aronsson:68}
\bysame, \emph{On the partial differential equation
  {$u_{x}{}^{2}\!u_{xx}+2u_{x}u_{y}u_{xy}+u_{y}{}^{2}\!u_{yy}=0$}}, Ark. Mat.
  \textbf{7} (1968), 395--425 (1968).

\bibitem{Aronsson:84}
\bysame, \emph{On certain singular solutions of the partial differential
  equation {$u^{2}_{x}u_{xx}+2u_{x}u_{y}u_{xy}+u^{2}_{y}u_{yy}=0$}},
  Manuscripta Math. \textbf{47} (1984), 133--151.

\bibitem{Bhattacharya-DiBenedetto-Manfredi:89}
T.~Bhattacharya, E.~DiBenedetto, and J.~Manfredi, \emph{Limits as
  {$p\to\infty$} of {$\Delta_pu_p=f$} and related extremal problems}, 1989,
  Some topics in nonlinear PDEs (Turin, 1989), pp.~15--68 (1991).

\bibitem{Crandall-Evans-Gariepy:01}
M.~G. Crandall, L.~C. Evans, and R.~F. Gariepy, \emph{Optimal {L}ipschitz
  extensions and the infinity {L}aplacian}, Calc. Var. Partial Differential
  Equations \textbf{13} (2001), 123--139.

\bibitem{Daskalopoulos-Uhlenbeck:22}
G.~Daskalopoulos and K.~Uhlenbeck, \emph{Analytic properties of stretch maps
  and geodesic laminations}, arXiv:2205.08250 [math.DG], 2022.

\bibitem{Daskalopoulos-Uhlenbeck:24.2}
\bysame, \emph{Best lipschitz maps and earthquakes}, arXiv:2410.08296
  [math.DG], 2024.

\bibitem{Daskalopoulos-Uhlenbeck:24}
\bysame, \emph{Transverse measures and best {L}ipschitz and least gradient
  maps}, J. Differential Geom. \textbf{127} (2024), 969--1018.

\bibitem{Duzaar-Mingione:04}
F.~Duzaar and G.~Mingione, \emph{The {$p$}-harmonic approximation and the
  regularity of {$p$}-harmonic maps}, Calc. Var. Partial Differential Equations
  \textbf{20} (2004), 235--256.

\bibitem{Eells-Lemaire:78}
J.~Eells and L.~Lemaire, \emph{A report on harmonic maps}, Bull. London Math.
  Soc. \textbf{10} (1978), 1--68.

\bibitem{Eells-Lemaire:88}
\bysame, \emph{Another report on harmonic maps}, Bull. London Math. Soc.
  \textbf{20} (1988), 385--524.

\bibitem{Evans:03.2}
L.~C. Evans, \emph{Three singular variational problems}, in Viscosity Solutions
  of Differential Equations and Related Topics, vol. 1323, Research Institute
  for the Matematical Sciences, RIMS Kokyuroku, 2003.

\bibitem{Evans-Savin:08}
L.~C. Evans and O.~Savin, \emph{{$C^{1,\alpha}$} regularity for infinity
  harmonic functions in two dimensions}, Calc. Var. Partial Differential
  Equations \textbf{32} (2008), 325--347.

\bibitem{Evans-Smart:11.2}
L.~C. Evans and C.~K. Smart, \emph{Everywhere differentiability of infinity
  harmonic functions}, Calc. Var. Partial Differential Equations \textbf{42}
  (2011), 289--299.

\bibitem{Evans-Yu:05}
L.~C. Evans and Y.~Yu, \emph{Various properties of solutions of the
  infinity-{L}aplacian equation}, Comm. Partial Differential Equations
  \textbf{30} (2005), 1401--1428.

\bibitem{Gallagher-Moser:23}
E.~Gallagher and R.~Moser, \emph{The {$\infty $}-elastica problem on a
  {R}iemannian manifold}, J. Geom. Anal. \textbf{33} (2023), Paper No. 226.

\bibitem{Gallagher-Moser:24}
\bysame, \emph{Weighted {$\infty$}-{W}illmore spheres}, NoDEA Nonlinear
  Differential Equations Appl. \textbf{31} (2024), Paper No. 55.

\bibitem{Grosse-Brauckmann:92}
K.~Gro{\ss}e-Brauckmann, \emph{Interior and boundary monotonicity formulas for
  stationary harmonic maps}, Manuscripta Math. \textbf{77} (1992), 89--95.

\bibitem{Hardt-Lin:87}
R.~Hardt and F.-H. Lin, \emph{Mappings minimizing the {$L^p$} norm of the
  gradient}, Comm. Pure Appl. Math. \textbf{40} (1987), 555--588.

\bibitem{Helein:91.1}
F.~H\'elein, \emph{R\'egularit\'e des applications faiblement harmoniques entre
  une surface et une vari\'et\'e riemannienne}, C. R. Acad. Sci. Paris S{\'e}r.
  I Math. \textbf{312} (1991), 591--596.

\bibitem{Horn-Johnson:85}
R.~A. Horn and C.~R. Johnson, \emph{Matrix analysis}, Cambridge University
  Press, Cambridge, 1985.

\bibitem{Hutchinson:86}
J.~E. Hutchinson, \emph{Second fundamental form for varifolds and the existence
  of surfaces minimising curvature}, Indiana Univ. Math. J. \textbf{35} (1986),
  45--71.

\bibitem{Ignat-Merlet:11}
R.~Ignat and B.~Merlet, \emph{Lower bound for the energy of {B}loch walls in
  micromagnetics}, Arch. Ration. Mech. Anal. \textbf{199} (2011), 369--406.

\bibitem{Ignat-Moser:25}
R.~Ignat and R.~Moser, \emph{Asymptotic minimality of one-dimensional
  transition profiles in {A}viles-{G}iga type models: an approach via
  $1$-currents}, arXiv:2508.13753 [math.AP], 2025.

\bibitem{Jensen:93}
R.~Jensen, \emph{Uniqueness of {L}ipschitz extensions: minimizing the sup norm
  of the gradient}, Arch. Rational Mech. Anal. \textbf{123} (1993), 51--74.

\bibitem{Katzourakis:12}
N.~Katzourakis, \emph{{$L^\infty$} variational problems for maps and the
  {A}ronsson {PDE} system}, J. Differential Equations \textbf{253} (2012),
  2123--2139.

\bibitem{Katzourakis:13}
\bysame, \emph{Explicit {$2D$} {$\infty$}-harmonic maps whose interfaces have
  junctions and corners}, C. R. Math. Acad. Sci. Paris \textbf{351} (2013),
  677--680.

\bibitem{Katzourakis:14.2}
\bysame, \emph{{$\infty$}-minimal submanifolds}, Proc. Amer. Math. Soc.
  \textbf{142} (2014), 2797--2811.

\bibitem{Katzourakis:14.1}
\bysame, \emph{On the structure of {$\infty$}-harmonic maps}, Comm. Partial
  Differential Equations \textbf{39} (2014), 2091--2124.

\bibitem{Katzourakis:15.2}
\bysame, \emph{Nonuniqueness in vector-valued calculus of variations in
  {$L^\infty$} and some linear elliptic systems}, Commun. Pure Appl. Anal.
  \textbf{14} (2015), 313--327.

\bibitem{Katzourakis:17.1}
\bysame, \emph{A characterisation of {$\infty$}-harmonic and {$p$}-harmonic
  maps via affine variations in {$L^\infty$}}, Electron. J. Differential
  Equations (2017), Paper No. 29.

\bibitem{Katzourakis-Moser:23}
N.~Katzourakis and R.~Moser, \emph{Variational problems in {$L^{\infty}$}
  involving semilinear second order differential operators}, ESAIM Control
  Optim. Calc. Var. \textbf{29} (2023), Paper No. 76.

\bibitem{Katzourakis-Moser:25}
\bysame, \emph{Minimisers of supremal functionals and mass-minimising
  1-currents}, Calc. Var. Partial Differential Equations \textbf{64} (2025),
  Paper No. 26.

\bibitem{Kruger:03}
A.~Ya. Kruger, \emph{On {F}r\'echet subdifferentials}, J. Math. Sci. (N.Y.)
  \textbf{116} (2003), 3325--3358, Optimization and related topics, 3.

\bibitem{Moser:15.2}
R.~Moser, \emph{Geroch monotonicity and the construction of weak solutions of
  the inverse mean curvature flow}, Asian J. Math. \textbf{19} (2015),
  357--376.

\bibitem{Moser:22}
\bysame, \emph{Structure and classification results for the $\infty$-elastica
  problem}, Amer.\ J.\ Math. \textbf{144} (2022), 1299--1329.

\bibitem{Nash:56}
J.~Nash, \emph{The imbedding problem for {R}iemannian manifolds}, Ann. of Math.
  (2) \textbf{63} (1956), 20--63.

\bibitem{Ou-Troutman-Wilhelm:12}
Y.-L. Ou, T.~Troutman, and F.~Wilhelm, \emph{Infinity-harmonic maps and
  morphisms}, Differential Geom. Appl. \textbf{30} (2012), 164--178.

\bibitem{Price:83}
P.~Price, \emph{A monotonicity formula for {Y}ang-{M}ills fields}, Manuscripta
  Math. \textbf{43} (1983), 131--166.

\bibitem{Rodriguez-Arenas-Wengenroth:24}
A.~Rodr{\'\i}guez-Arenas and J.~Wengenroth, \emph{Smirnov-decompositions of
  vector fields}, arXiv:2405.13406 [math.FA], 2024.

\bibitem{Savin:05}
O.~Savin, \emph{{$C^1$} regularity for infinity harmonic functions in two
  dimensions}, Arch. Ration. Mech. Anal. \textbf{176} (2005), 351--361.

\bibitem{Sheffield-Smart:12}
S.~Sheffield and C.~K. Smart, \emph{Vector-valued optimal {L}ipschitz
  extensions}, Comm. Pure Appl. Math. \textbf{65} (2012), 128--154.

\bibitem{Simon:83}
L.~Simon, \emph{Lectures on geometric measure theory}, Australian National
  University Centre for Mathematical Analysis, Canberra, 1983.

\bibitem{Simon:96}
\bysame, \emph{Theorems on regularity and singularity of energy minimizing
  maps}, Lectures in Math.~ETH Z\"urich, Birkh\"auser, Basel, 1996.

\bibitem{Smirnov:93}
S.~K. Smirnov, \emph{Decomposition of solenoidal vector charges into elementary
  solenoids, and the structure of normal one-dimensional flows}, Algebra i
  Analiz \textbf{5} (1993), 206--238.

\bibitem{Thurston:98}
W.~P. Thurston, \emph{Minimal stretch maps between hyperbolic surfaces},
  arXiv:math/9801039 [math.GT], 1998.

\bibitem{Troutman:08}
T.~L. Troutman, \emph{Infinity-harmonic functions, maps, and morphisms of
  {R}iemannian manifolds}, University of California, Riverside, 2008, Thesis
  (Ph.D.).

\bibitem{Wang-Ou:09}
Z.-P. Wang and Y.-L. Ou, \emph{Classifications of some special
  infinity-harmonic maps}, Balkan J. Geom. Appl. \textbf{14} (2009), 120--131.

\bibitem{Wei:08}
S.~W. Wei, \emph{{$p$}-harmonic geometry and related topics}, Bull. Transilv.
  Univ. Bra\c sov Ser. III \textbf{1(50)} (2008), 415--453.

\end{thebibliography}
\end{document}